\documentclass[12pt]{article}
\usepackage{latexsym, amsmath, amsfonts, amsthm, amssymb}
\usepackage{mathtools}
\usepackage{a4wide}
\usepackage{times}
\usepackage{graphicx, tocloft}
\usepackage{mathrsfs}
\usepackage{bbm}
\usepackage{hyperref}
\usepackage{geometry}

\usepackage{xcolor}

\def \qqquad {\qquad \qquad }
\def \RR {\mathbb R}
\def \NN {\mathbb N}
\def \EE {\mathbb E}

\def \PP {\mathbb P}

\def \one {\mathbbm 1}

\def \cD {\mathcal D}

\def \cX {\mathcal X}

\def \ind {\mathbbm{1}}

\def \cE {\mathcal E}
\def \cF {\mathcal F}

\def \cA {\mathcal A}
\def \cS {\mathcal S}
\def \cB {\mathcal B}

\def \cD {\mathcal D}

\def \cR {\mathcal R}

\def \cP {\mathcal P}
\def \cN {\mathcal N}

\def \cC {\mathcal C}

\newcommand\numberthis{\addtocounter{equation}{1}\tag{\theequation}}

\newcommand{\Var}{\mathrm{Var}}
\newcommand{\Cov}{\mathrm{Cov}}
\newcommand{\Ent}{\mathrm{Ent}}
\newcommand{\DKL}{\mathrm{D_{KL}}}

\newcommand{\normgamma}[1]{\lVert #1 \rVert_{L^2(\gamma)}}
\newcommand{\normmu}[1]{\lVert #1 \rVert_{L^2(\mu)}}
\DeclareMathOperator*{\argmin}{\arg\!\min}

\DeclarePairedDelimiter\abs{\lvert}{\rvert}%
\DeclarePairedDelimiter\norm{\lVert}{\rVert}%

\newtheorem{theorem}{Theorem}[section]

\newtheorem{lemma}[theorem]{Lemma}

\newtheorem{proposition}[theorem]{Proposition}
\newtheorem{corollary}[theorem]{Corollary}

\newtheorem*{corollary*}{Corollary}

\theoremstyle{definition}
\newtheorem{defi}[theorem]{Definition}
\newtheorem{remark}[theorem]{Remark}

\def\qed{\hfill $\vcenter{\hrule height .3mm
		\hbox {\vrule width .3mm height 2.1mm \kern 2mm \vrule width .3mm
			height 2.1mm} \hrule height .3mm}$ \bigskip}

\def \id {{\rm Id}}

\def \id {{ \rm Id}}

\def \cov {{ \rm Cov}\,}

\title{Entropy and Learning of Lipschitz Functions under Log-Concave Measures}
\author{Pierre Bizeul\textsuperscript{1}
and Boaz Klartag\textsuperscript{2}}

\date{}

\begin{document}

\maketitle

\footnotetext[1]{Department of Mathematics, Weizmann Institute of Science, Rehovot 7610001, Israel. Email: \href{mailto:pierre.bizeul@weizmann.ac.il}{pierre.bizeul@weizmann.ac.il}. Partially supported by the European Research Council (ERC) under the European Union’s Horizon 2020 research and innovation programme, grant agreement No. 101001677 “ISOPERIMETRY”.}

\footnotetext[2]{Department of Mathematics, Weizmann Institute of Science, Rehovot 7610001, Israel. Email: \href{mailto:boaz.klartag@weizmann.ac.il}{boaz.klartag@weizmann.ac.il}. Supported by a grant from the Israel Science Foundation (ISF).}

\begin{abstract}
We study regression of $1$-Lipschitz functions under a
log-concave measure $\mu$ on $\RR^d$. We focus on the high-dimensional regime where the sample size $n$ is subexponential in $d$, in which distribution-free estimators are ineffective. We analyze two polynomial-based procedures: the projection estimator, which relies on knowledge of an orthogonal polynomial basis of $\mu$, and the least-squares estimator over low-degree polynomials, which requires no knowledge of $\mu$ whatsoever. Their risk is governed by the rate of polynomial approximation of Lipschitz functions in $L^2(\mu)$. When this rate matches the Gaussian one, we show that both estimators achieve minimax bounds over a wide range of parameters. A key ingredient is sharp entropy estimates for the class of $1$-Lipschitz functions in $L^2(\mu)$, which are new even in the Gaussian setting.
\end{abstract}

\section{Introduction}

In this paper, we study the following regression problem. Given an unknown $1$-Lipschitz function $f: \RR^d \rightarrow \RR$, we observe data
\[
\left((X_1, Y_1), \dots, (X_n, Y_n)\right),
\]
where:
\begin{itemize}
\item The vectors $X_1,\ldots,X_n \in \RR^d$ are independent random vectors that are distributed according to some Borel probability measure $\mu$ on $\RR^d$ that may or may not be known to us.
\item The numbers $Y_1,\ldots,Y_n \in \RR$ are noisy observations of the function $f$ evaluated at $X_i$, that is,
\begin{equation}\label{eq138}
    Y_i = f(X_i) + \xi_i, \qquad i = 1, \ldots, n,
\end{equation}
where, throughout the paper, $\xi_1,\ldots,\xi_n$ are independent, real-valued Gaussian random variables of mean zero and variance $\sigma^2$, for some parameter $\sigma > 0$.
\end{itemize}

Our goal is to construct an estimator $\hat{f}: \RR^d \rightarrow \RR$ of the function $f$, whose performance is measured by the $L^2(\mu)$-risk, defined via
\begin{equation}
    \cR(\hat{f},f) := \EE \norm{f - \hat{f}}^2_{L^2(\mu)}.
\end{equation}

 There are various types
of probability measures $\mu$ for which our analysis applies. We first consider a relatively simple case:

\subsection{The Gaussian case}
\label{sec11}

Consider first the case where $\mu = \gamma = \gamma_d$, the standard Gaussian measure on $\RR^d$. A well-known fact (recalled below) is that any $1$-Lipschitz function
$f: \RR^d \rightarrow \RR$ can be approximated by polynomials in Gaussian space. Namely, for any $m \geq 1$, there exists a polynomial $P_m: \RR^d \rightarrow \RR$ of total degree at most $m$ such that
\begin{equation}\label{eq146}
    \normgamma{f - P_m}^2 \leq \frac{1}{m+1}.
\end{equation}
Here and throughout the paper, the degree of a multivariate polynomial refers to its total degree. More precisely, for a multi-index $\alpha = (\alpha_1,\ldots,\alpha_d) \in \NN^d$ and the corresponding monomial
\[
P(x) = \prod_{i=1}^d x_i^{\alpha_i} \qquad \qquad \qquad x = (x_1,\ldots,x_d) \in \RR^d,
\]
we define
\[
\deg(P) := \sum_{i=1}^d \alpha_i =: |\alpha|.
\]
Here $\NN = \{0,1,2,\ldots \}$ stands for the set of all non-negative integers.
The degree of a multivariate polynomial is the maximum of the degrees of its monomials.
Note that the polynomial $P_m$ in \eqref{eq146} is simply the orthogonal projection of $f$ onto the finite-dimensional space of polynomials on $\RR^d$ of degree at most $m$, denoted by $\cP_{d,m}$. In particular, denoting by
\[
(H_\alpha)_{\alpha \in \NN^d}
\]
the Hermite basis of orthogonal polynomials for $\gamma$, one can write
\begin{equation}
P_m = \sum_{\alpha \in \NN^d, | \alpha | \leq m} \langle f, H_\alpha \rangle_{L^2(\gamma)} \, H_\alpha.
\label{eq_454}
\end{equation}
Our goal is to construct an estimator for the function $f$. Thanks to the polynomial approximation property  \eqref{eq146},
a natural idea is to estimate the polynomial $P_m \in \cP_{d,m}$, for a suitable choice of degree $m$ depending on $n, d$ and $\sigma$. This reduces
the nonparametric problem \eqref{eq138} to a parametric one.
In view of (\ref{eq_454}), for a well-chosen $m$, one may construct an estimator $\hat{f}$ by empirically estimating the coefficients
\[
f_\alpha := \langle f, H_\alpha \rangle_{L^2(\gamma)}.
\]
Namely, we define
\begin{equation}
    \hat{f} := \sum_{\alpha \in \NN^d, | \alpha | \leq m} \hat{f}_\alpha \, H_\alpha,
\end{equation}
where the coefficients $(\hat{f}_\alpha)_{| \alpha | \leq m}$ are defined as follows:
\begin{itemize}
\item First, for $\alpha = 0$, we estimate the Gaussian integral of $f$ (its ``barycenter'')
$$ a := f_0 = \int_{\RR^d} f d \gamma $$
via
\begin{equation}
    \hat a := \hat{f}_0 = \frac{1}{n}\sum_{i=1}^n Y_i = \frac{1}{n}\sum_{i=1}^n f(X_i) + \frac{1}{n}\sum_{i=1}^n \xi_i.
\end{equation}
Clearly $\hat a$ is an unbiased estimator of $a$.
\item Next, for any $\alpha \in \NN^d$ with $| \alpha | > 0$ we define
\begin{align}
    \hat{f}_\alpha &= \frac{1}{n}\sum_{i=1}^n (Y_i-\hat{a}) H_\alpha(X_i)\label{eq166}\\
    &= \frac{1}{n}\sum_{i=1}^n (f(X_i)-a) H_\alpha(X_i) + \frac{1}{n}\sum_{i=1}^n \xi_i H_\alpha(X_i) + \frac{1}{n}\sum_{i=1}^n (\hat a -a)H_\alpha(X_i),
\end{align}
which is a biased estimator of $f_\alpha$.
\end{itemize}

Note that the na\"ive unbiased estimator of $f_\alpha$, namely
\begin{equation}\label{eq214}
    \check{f}_\alpha = \frac{1}{n}\sum_{i=1}^n Y_i H_\alpha(X_i),
\end{equation}
may have an arbitrarily large variance, since we make no assumptions on the barycenter of $f$.
If one assumes that the barycenter of $f$ lies in some ball of fixed radius, independent of the dimension $d$ and of the sample size $n$, then
it makes sense to use the simpler estimator $\check{f}$ in place of $\hat{f}$.

\medskip Up to this minor variance reduction procedure, the estimator $\hat{f}$ is simply the projection estimator of $f$ in the orthonormal basis of Hermite polynomials $(H_\alpha)_{\alpha \in \NN^d}$.

\subsection{The log-concave case}

Moving away from the Gaussian setting, we aim to generalize the learning procedure from Section \ref{sec11} to other measures. We shall assume that:
\begin{itemize}
    \item The probability measure $\mu$ is a log-concave measure on $\RR^d$, meaning that
    \[
    d\mu(x) = e^{-V(x)}\, dx
    \]
    for some convex potential $V: \RR^d \to \RR \cup \{\infty\}$;

    \item The probability measure $\mu$ satisfies a polynomial approximation property: for any $1$-Lipschitz function $f: \RR^d \rightarrow \RR$ and an integer $m \geq 1$, there exists a polynomial $P_m: \RR^d \rightarrow \RR$ of degree at most $m$ such that
    \begin{equation}\label{eq177}
        \norm{f - P_m}_{L^2(\mu)}^2 \leq \Psi^2_\mu(m),
    \end{equation}
    for some function $\Psi_\mu: \NN \to \RR^+$ decreasing to $0$ as $m \to \infty$;

    \item for normalization purposes, let us assume that
    \begin{equation}\label{eq182}
        \Psi_\mu(0) = 1.
    \end{equation}
    In other words, for any $1$-Lipschitz function $f$,
    \[
    \Var_\mu(f) = \norm{f - \EE_\mu f}_{L^2(\mu)}^2 \leq \Psi_\mu^2(0) = 1.
    \]
\end{itemize}

A probability measure $\mu$ on $\RR^d$ with finite second moments is {\it isotropic} if $\int_{\RR^d} x_i d \mu(x) = 0$ for all $i$, and
$\cov(\mu) = \id$, where $\cov(\mu) = (\cov_{ij}(\mu))_{i,j=1,\ldots,n} \in \RR^{n \times n}$ is the covariance matrix, defined via
$$ \cov_{ij}(\mu) = \int_{\RR^d} x_i x_j d \mu(x) - \int_{\RR^d} x_i d \mu(x) \int_{\RR^d} x_j d \mu(x). $$
Below we will mostly work with the isotropic normalization. The {\it projection estimator} is defined as follows:

\begin{defi}
    Let $\mu$ be an isotropic log-concave measure on $\RR^d$. Let $(P_\alpha)_{\alpha\in \NN^d}$ be an orthonormal basis of polynomials in $L^2(\mu)$ with $\deg(P_\alpha) = | \alpha |$ for all $\alpha$. Given observations of the form \eqref{eq138} and some parameter $m\in\NN$, we define the projection estimator by
    \begin{equation}
        \hat{f} := \sum_{\alpha, \deg(P_\alpha) \leq m} \hat{f}_\alpha \, P_\alpha,
    \end{equation}
    where
    \begin{equation}
        \hat a := \hat{f}_0 = \frac{1}{n}\sum_{i=1}^n Y_i,
    \end{equation}
    and for all the other coefficients
    \begin{equation}\label{eq260}
        \hat{f}_\alpha = \frac{1}{n}\sum_{i=1}^n (Y_i-\hat a) P_\alpha(X_i).
    \end{equation}
\end{defi}

Let us mention that the Kannan–Lovász–Simonovits (KLS) conjecture suggests that the normalization \eqref{eq182} is equivalent to normalizing the largest variance over all directions:
\begin{equation}\label{eq189}
    c \leq \norm{\Cov(\mu)}_{op} \leq 1
\end{equation}
for some universal constant $c > 0$, where $\| \cdot \|_{op}$ is the operator norm. For two functions $a$ and $b$, we write
$a \lesssim b$ or $a = O(b)$ if there exists a universal constant $C>0$ such that $a\leq Cb$. We write $a\simeq b$ if $a\lesssim b$ and $b \lesssim a$. Using the best current bounds on the KLS constant \cite{klartag2023logarithmic}, one can take $c = c_n \simeq 1/\log n$ in \eqref{eq189}.

\medskip
Log-concave measures provide a natural generalization of the Gaussian case for two reasons. First, the behavior of Lipschitz and polynomial functions of a log-concave random vector is relatively well-understood. Second, although few explicit bounds are known, the polynomial approximation property \eqref{eq177} always holds—albeit possibly with a slowly decaying function $\Psi_\mu$. A detailed discussion of these facts is provided in Section~\ref{sec_logconcave}.

\medskip
We prove the following upper bound on the performance of the {\it projection estimator}.

\begin{theorem}\label{thm_proj}
Let $n,d \geq 2$, and assume that the variance of the noise $\sigma^2$ satisfies
    \[
        \sigma^2\leq d.
    \]
    Define
    \[
        m_0 = \lfloor \tfrac{\log n}{\log d}\rfloor.
    \]
    We distinguish between two regimes:
    \begin{itemize}
        \item If $d^5\leq n\leq e^{\sqrt{d}\log d}$, set
        \[
            m := m_0-4.
        \]
        For this choice of degree $m$ we obtain the bound
        \begin{equation}\label{eq259}
            \EE \normmu{f-\hat{f}}^2 \leq \Psi^2_\mu(m) + O\!\left(\tfrac{1}{d}\right).
        \end{equation}
        \item If $e^{\sqrt{d}\log d}\leq n\leq e^{d\log d/2}$, set
        \[
            m = m_0 - \Bigl\lceil\frac{4\log m_0}{\log (d/m_0)}\Bigr\rceil.
        \]
        For this choice of degree $m$ we obtain the bound
        \begin{equation}\label{eq268}
            \EE \normmu{f-\hat{f}}^2 \leq \Psi^2_\mu(m) + O\!\left(\tfrac{1}{m^2}\right).
        \end{equation}
    \end{itemize}
\end{theorem}

The computation of the projection estimator requires apriori knowledge of an orthonormal basis of polynomials for $\mu$.
In the more general setting where $\mu$ is an {\it unknown} log-concave probability measure, one may instead use the polynomial that minimizes the empirical least-squares error.

\begin{defi}
     Let $\mu$ be an isotropic log-concave measure. Given observations of the form \eqref{eq138} and some parameter $m\in\NN$, we define the {\it least-squares estimator} by
\end{defi}
\begin{equation} \label{eq_127}
    \hat{f}_{\mathrm{LS}} := \argmin_{\deg(P) \leq m} \sum_{i=1}^n (P(X_i) - Y_i)^2,
\end{equation}
That is, the sum on the right-hand side of (\ref{eq_127}) is a quadratic function on the finite-dimensional space $\cP_{d,m}$ of polynomials of degree
at most $m$ on $\RR^d$, and we define the estimator $\hat{f}_{\mathrm{LS}}$ to be any minimizer of this quadratic function.
Note that the computation of the {\it least-squares estimator} requires no knowledge about the underlying measure $\mu$.

\medskip We show that the performance of the least-squares estimator $\hat{f}_{\mathrm{LS}}$ is comparable to that of the projection estimator $\hat{f}$ in certain regimes.

\begin{theorem}\label{thm_ls}
Let $n,d \geq 2$, and assume that the variance of the noise $\sigma^2$ satisfies
    \[
        \sigma^2\leq d.
    \]
    Define
    \[
        m_0 = \lfloor\tfrac{\log n}{\log d}\rfloor.
    \]
    There exist universal constants $c_0,C_0>0$ such that the following hold:
    \begin{itemize}
        \item If
        \[
            d^5\leq n \leq e^{c_0\log^2 d/\log\log d},
        \]
        set $m = m_0-4$. For this choice of degree $m$ we have the bound,
        \begin{equation}\label{eq294}
            \EE \normmu{f-f_{LS}}^2 \leq \Psi^2_\mu(m) + O\!\left(\tfrac{1}{d}\right).
        \end{equation}
        \item If
        \[
            e^{c_0\log^2 d/\log\log d}\leq n \leq e^{\tfrac{d^{\beta}}{C_0}},
        \]
        for some $\beta <1/2$, define
        \[
            \alpha = \frac{\log(C_0\log n)}{\log d}<\tfrac{1}{2}, \qquad
            m = m_0 - 4 - \lfloor2\alpha m_0\rfloor.
        \]
        For this choice of degree $m$, assuming that $d\geq d(\beta)$ so that $m\geq0$,
        \begin{equation}\label{eq295}
            \EE \normmu{f-f_{LS}}^2 \leq \Psi^2_\mu(m) + O\!\left(\tfrac{1}{d}\right).
        \end{equation}
    \end{itemize}
\end{theorem}

We also provide here lower bounds for the minimax rate of the learning problem \eqref{eq138}. For a fixed probability measure $\mu$ on $\RR^d$, define the minimax rate
\begin{equation}
    \cR^*_{n,d} = \inf_{\tilde{f}}\sup_{f}\cR(f,\tilde{f}),
\end{equation}
where the infimum runs over all estimators $\tilde{f}$ (i.e., all measurable functions of the data $(X_i,Y_i)_{i=1}^n$) and the supremum runs over all $1$-Lipschitz functions $f$. A standard information-theoretic way of providing a lower bound on $\cR^*_{n,d}$ is the Fano method \cite{wainwright2019high}, which requires entropy estimates. More precisely, let
\[
d(f,g) := d_\mu(f,g) = \normmu{f-g},
\]
and let
\[
B_{Lip} = \left \{f: \RR^d \rightarrow \RR \ : \  f \textrm{ is } 1\textrm{-Lipschitz with }  \int f^2 d\mu \leq 1 \right \}
\]
be the unit ball of $1$-Lipschitz functions for this metric. For $0<\varepsilon<1$, define
\[
\cN(B_{Lip},\varepsilon, d_\mu)
\]
to be the maximal size of an $\varepsilon$-separated set in $B_{Lip}$ with respect to the metric $d = d_{\mu}$, and set
\[
H_L^\mu(\varepsilon) = \log \cN(B_{Lip},\varepsilon, d_\mu),
\]
the entropy of Lipschitz functions with respect to $d_\mu$.

\medskip
We lower bound $H_L^\mu$ when $\mu$ is an isotropic log-concave measure, with an improvement when it has a product structure. We say that a probability measure $\mu$ on $\RR^d$
is a product measure if $X_1,\ldots,X_d$ are independent random variables whenever $X = (X_1,\ldots,X_d)$ has law $\mu$.

\begin{theorem}\label{thm_entropyLip}
    Let $\mu$ be an isotropic log-concave measure on $\RR^d$. Then for any $\varepsilon$ with
    \[
        d^{-\eta} < \varepsilon < 1,
    \]
    we have
    \begin{equation}\label{eq388}
        \binom{d}{\lfloor c/\varepsilon \rfloor^2} \ \lesssim \ H_L^{\mu}(\varepsilon),
    \end{equation}
    where $\eta<1/4$ and $c>0$ are universal constants.
    Moreover, if additionally $\mu$ is a product measure, then  \eqref{eq388} holds with $\eta=1/4$, that is, it holds  in the range
    \[
    d^{-1/4} < \varepsilon < 1.
    \]
\end{theorem}
As we will see in Section \ref{sec_entropy}, the estimate \eqref{eq388} is tight up to the value of the constant $c$.
Note that it is more conventional to define entropy via covering numbers rather than packing numbers.
Since the two definitions are equivalent up to a factor of $2$ in $\varepsilon$, this choice does not affect the result.
Note that it is more conventional to define entropy via covering numbers rather than packing numbers.
Since the two definitions are equivalent up to a factor of $2$ in $\varepsilon$, this choice does not affect the result.

\medskip To the best of our knowledge, this result is new even in the Gaussian setting, and might be of independent interest. It allows us to derive minimax lower bounds for the learning problem \eqref{eq177}.

\begin{corollary}\label{corollary_fano}
     Let $\mu$ be an isotropic log-concave measure on $\RR^d$. Assume that the noise satisfies
     \[
         n^{-\kappa} \leq \sigma^2 \leq n
     \]
     for some constant $\kappa>0$. There exists a universal constant $c>0$ such that if
     \[
         n \leq e^{\frac{cd^{2\eta}\log d}{\kappa}},
     \]
     the minimax risk is lower bounded as
     \begin{equation}\label{eq415}
         \cR^*_{n,d}\gtrsim (1+\kappa)\frac{\log n}{\log d}.
     \end{equation}
      Moreover, if additionally $\mu$ is a product measure, then the lower bound \eqref{eq415} holds in the range
     \[
         n \leq e^{\frac{c\sqrt{d}\log d}{\kappa}}.
     \]
\end{corollary}

Thus, in the Gaussian case, or more generally, when $\mu$ is an isotropic log-concave measure satisfying
\begin{equation}\label{eq279}
    \Psi^2_\mu(m)\lesssim \frac{1}{m},
\end{equation}
we obtain matching bounds in certain regimes for both the projection and least-squares estimators. Specializing the previous bounds to the case where, say, $\kappa = 10$ we obtain the following:

\begin{corollary} Let $n,d \geq 2$, and let $\mu$ be an isotropic log-concave measure on $\RR^d$ satisfying \eqref{eq279}, such as the Gaussian measure or the uniform measure on the hypercube. Assume, for instance, that the noise parameter $\sigma > 0$ satisfies
     \[
         \frac{1}{n^{10}} \leq \sigma^2 \leq d.
     \]
     Then the following hold:
     \begin{itemize}
         \item The projection estimator and the least squares estimators achieves the minimax rate, up to a universal constant, in the range
         \[
             d^5\leq n \leq e^{cd^{2\eta}\log d},
         \]
         where $c > 0$ is a universal constant. That is,
         \begin{equation}\label{eq480}
             \frac{\log d }{\log n} \lesssim \cR^*_{n,d} \leq \cR(f,\hat{f}) \simeq \cR(f,f_{LS}) \lesssim \frac{\log d}{\log n}.
         \end{equation}
         \item If $\mu$ is additionally a product measure, then the projection estimator achieves minimax rate in the larger range
        \[
             d^5\leq n \leq e^{c\sqrt{d}\log d}.
         \]
         For the least square estimator, there exists a universal constant $C > 0$ such that
         for any  $0<\beta<1/2$, if
         \[
             d^5 \leq n \leq e^{d^{\beta}/C},
         \]
         and $d\geq d(\beta)$ then the minimax rate is achieved up to a factor $(1-2\beta)^{-1}$:
         \[
             \frac{\log d }{\log n} \lesssim \cR^*_{n,d} \leq \cR(f,f_{LS}) \lesssim (1-2\beta)^{-1}\frac{\log d}{\log n}.
         \]
     \end{itemize}
\end{corollary}

In comparison, typical regression algorithms for smooth functions -- such as nearest neighbors -- require a number of samples that is at least exponential in the dimension.   In contrast, our proposed algorithms attains the minimax rate in the high-dimensional regime, when the number of samples is merely subexponential in the dimension. As a concrete takeaway, consider learning a $1$-Lipschitz function from noisy observations in $L^2(\gamma)$, where we recall that $\gamma = \gamma_d$ is the standard Gaussian measure in $\RR^d$. In order to achieve accuracy up to a factor $\varepsilon>0$, it suffices to use a sample size that grows only polynomially with the dimension:
\[
    n \simeq d^{\frac{c}{\varepsilon}}
\]
for some constant $c>0$. To the best of our knowledge, this result is new already in the Gaussian case. Our approach is related to the recent works of Eskenazis, Ivanishvili and Streck (\cite{eskenazis2022learning}, \cite{eskenazis2022low}) on learning over the discrete hypercube, which rely on expansions in the orthonormal Walsh polynomial basis.

\medskip
The remainder of this paper is organized as follows:

\medskip
In Section~\ref{sec_logconcave}, we review several properties of log-concave measures that will be used throughout the paper. We recall concentration inequalities for Lipschitz and polynomial functions, and present polynomial approximation results for Lipschitz functions in $L^2(\mu)$. This background sets the stage for the statistical analysis.

\medskip
In Section~\ref{sec_algo}, we study in detail the two algorithms proposed for estimating Lipschitz functions: the projection estimator and the least-squares estimator. For both procedures we establish upper bounds on their $L^2(\mu)$ risk (Theorems~\ref{thm_proj} and \ref{thm_ls}).

\medskip
In Section~\ref{sec_entropy}, we turn to lower bounds. We provide new estimates on the metric entropy of the class of $1$-Lipschitz functions under isotropic log-concave measures (Theorem~\ref{thm_entropyLip}). As a consequence, we derive minimax lower bounds for the regression problem \eqref{eq138}, showing that in certain regimes the upper and lower bounds match (Corollary~\ref{corollary_fano}).

\section{Background on log-concave measures}\label{sec_logconcave}
In this section, we recall several properties of log-concave measures that are central to our study. We begin with the concentration properties of Lipschitz and polynomial functions under log-concave distributions. We then briefly review key facts about the Langevin semigroup associated with a log-concave measure. Finally, we discuss polynomial approximation of Lipschitz functions with respect to such measures.

\subsection{Concentration of Lipschitz and polynomial functions}
We recall that a probability measure $\mu$ on $\RR^d$ is log-concave if it takes the form
\begin{equation}
\label{eq_313}
    \mu(dx) = e^{-V(x)}\,dx
\end{equation}
for some convex function $V: \RR^d \to \RR \cup \{\infty\}$. If the measure is supported in an affine subspace of $\RR^d$,
we require that its density relative to this affine subspace will take the form (\ref{eq_313}) for some convex function $V$.
Gaussian measures, uniform distributions on convex bodies, and Dirac measures are all examples of log-concave probabilities. The convexity of $V$ is known to imply strong concentration properties for $\mu$.

\medskip
We say that $\mu$ satisfies a Poincar\'{e} inequality with constant $C > 0$ if, for all locally Lipschitz functions $f$,
\begin{equation}\label{eq226}
    \Var_\mu(f) \leq C \int \abs{\nabla f}^2 \, d\mu.
\end{equation}
The best constant $C > 0$ in the Poincar\'{e} inequality is denoted by $C_P(\mu)$ and referred to as the \emph{Poincar\'{e} constant} of $\mu$. Namely,
\[
C_P(\mu) := \sup_f \frac{\Var_\mu(f)}{\int \abs{\nabla f}^2 \, d\mu},
\]
where the supremum is taken over all locally Lipschitz, non-constant functions $f$. We also define
\[
C_P^{\mathrm{Lip}}(\mu) := \sup_{\substack{f\in \mathrm{Lip}_1}} \Var_\mu(f),
\]
where the supremum is over all $1$-Lipschitz functions $f$.
Our normalization assumption \eqref{eq182} rewrites as
\[
C_P^{\mathrm{Lip}}(\mu) = \Psi_\mu(0) = 1.
\]
It is clear from the definitions that
\[
C_P^{\mathrm{Lip}}(\mu) \leq C_P(\mu).
\]
However, a theorem of Emanuel Milman \cite{milman2009role} asserts that, for log-concave measures, these two quantities are equivalent up to a universal constant:
\begin{equation}
    C_P(\mu) \simeq C_P^{\mathrm{Lip}}(\mu).
\end{equation}
In other words, for log-concave measures, the Poincar\'{e} inequality is essentially saturated by Lipschitz functions.
The KLS conjecture, originally formulated in \cite{kannan1995isoperimetric}, proposes an even stronger statement: that the Poincar\'{e} inequality is actually saturated by linear functions. Namely, the trivial chain of inequalities
\[
\norm{\Cov(\mu)}_{\mathrm{op}} \leq C_P^{\mathrm{Lip}}(\mu) \leq C_P(\mu)
\]
could be reversed up to universal constants. The best known estimate to date is due to \cite{klartag2023logarithmic}:
\begin{equation}
    C_P(\mu) \lesssim \log n.
\end{equation}

A related but stronger functional inequality is the logarithmic Sobolev inequality. We say that $\mu$ satisfies a log-Sobolev inequality with constant $\rho > 0$ if, for all locally Lipschitz functions $f$,
\begin{equation}\label{493}
    \Ent_\mu(f^2) \leq 2\rho \int \abs{\nabla f}^2 \, d\mu.
\end{equation}
The best constant $\rho > 0$ for which this holds is denoted by $\rho_{LS}(\mu)$ and referred to as the \emph{log-Sobolev constant} of $\mu$.
It always holds that
\[
C_P(\mu) \leq \rho_{LS}(\mu),
\]
and the log-Sobolev inequality is strictly stronger than the Poincar\'{e} inequality, as it implies subgaussian concentration rather than merely subexponential (see Proposition~\ref{prop:lipschitz_concentration} below). In particular, unlike the Poincar\'{e} inequality, not all log-concave measures satisfy a log-Sobolev inequality. We refer to \cite{bizeul2023log} for further details.
As a central example, the standard Gaussian measure satisfies
\[
C_P(\gamma) = \rho_{LS}(\gamma) = 1.
\]

As mentioned before, a Poincar\'{e} inequality implies exponential concentration for Lipschitz functions, while a log-Sobolev inequality implies stronger subgaussian concentration. These facts were observed by Gromov–Milman \cite{gromov1983topological} (for Poincar\'{e}) and Herbst (for log-Sobolev), and are summarized in the following proposition.

\begin{proposition}\label{prop:lipschitz_concentration}
Let $\mu$ be a probability measure on $\RR^d$ and $f$ a $1$-Lipschitz function. There exists a universal constant $C > 0$ such that, for any $p \geq 1$,
\[
\norm{f}_{L^p(\mu)} \leq C p \sqrt{C_P(\mu)},
\]
\[
\norm{f}_{L^p(\mu)} \leq C \sqrt{p} \sqrt{\rho_{LS}(\mu)}.
\]
\end{proposition}

In particular, under our normalization \eqref{eq182}, the moments of a Lipschitz function grow at most linearly with $p$. This fact can be reformulated in the context of Orlicz norms. A random variable \( X \) is said to be \emph{sub-exponential} if there exists \( K > 0 \) such that
\[
\mathbb{E}[\exp(|X|/K)] \leq 2,
\]
and \emph{sub-Gaussian} if
\[
\mathbb{E}[\exp(X^2/K^2)] \leq 2.
\]
The smallest such constant \( K \) defines the Orlicz norms \( \|X\|_{\psi_1} \) and \( \|X\|_{\psi_2} \) respectively. A well-known equivalent definition of the Orlicz norm is $$ \| X \|_{\psi_{\alpha}} \simeq \sup_{m \geq 1} \frac{\| X \|_{m}}{m^{1/\alpha}}.  $$
Proposition \ref{prop:lipschitz_concentration} can be reformulated as :
\[
\|f(X)\|_{\psi_1} \leq C \sqrt{C_P(\mu)}, \quad \text{and} \quad \|f(X)\|_{\psi_2} \leq C \sqrt{\rho_{LS}(\mu)}.
\]
We recall the following standard Bernstein-type inequalities:
\begin{proposition}\label{prop:bernstein} Let $X_1,\ldots,X_n$ be independent centered random variables. Then
\[
\left\|\frac{1}{n}\sum_{i=1}^n X_i\right\|_{\psi_1} \lesssim \frac{1}{\sqrt{n}} \max_i \|X_i\|_{\psi_1}, \quad
\left\|\frac{1}{n}\sum_{i=1}^n X_i\right\|_{\psi_2} \lesssim \frac{1}{\sqrt{n}} \max_i \|X_i\|_{\psi_2},
\]
\end{proposition}
We refer to \cite{vershynin2018high} for background on subexponential and subgaussian distributions.
As for polynomials, when the underlying measure is log-concave, we have the following estimates.
\begin{proposition}\label{prop:poly_concentration_logconcave}
Let $\mu$ be a log-concave measure on $\RR^d$, and let $P$ be a degree-$m$ polynomial.
Abbreviate $\norm{P}_q = \| P \|_{L^q(\mu)}$.
Then there exists a universal constant $C > 0$ such that, for any $q \geq 2$,
\[
\norm{P}_q \leq \min \left( C^{(q-2)m},\, (Cq)^m \right) \norm{P}_2.
\]
Moreover, for any $q \geq 1$, there exists $C_1 > 0$ such that
\[
\norm{P}_q \leq (C_1 q)^m \norm{P}_1.
\]
\end{proposition}

\begin{proof}
The inequality
\[
\norm{P}_q \leq (C_1 q)^m \norm{P}_2 \leq (C_2 q)^m \norm{P}_1
\]
was essentially established by Bourgain \cite{bourgain2006distribution}, see also \cite{NSV}. It remains to interpolate for $q$ close to $2$. We may assume without loss of generality that $\norm{P}_2 = 1$. By Hölder's inequality, for $2 \leq q \leq 4$,
\[
\norm{P}_q^q \leq \norm{P}_2^{4 - q} \norm{P}_4^{q - 2} \leq C^{m(q - 2)},
\]
which concludes the proof.
\end{proof}
We remark that the following improvement holds when $\mu = \gamma$, the standard Gaussian measure (see \cite[Proposition 5.48]{aubrun2017alice}):

\begin{lemma}
Let $P$ be a degree-$m$ polynomial on $\RR^d$. Then
\[
\norm{P}_{L^q(\gamma)} \leq (q - 1)^{m/2} \norm{P}_{L^2(\gamma)}.
\]
\end{lemma}

We will also need a classical anti-concentration result for polynomials in log-concave variables, due to Carbery and Wright (\cite[Theorem 8]{carbery2001distributional})

\begin{theorem}\label{thm:carbery}
Let $X$ be a log-concave random vector in $\RR^d$, and let $P$ be a polynomial of degree at most $m$ such that $\EE P^2(X) = 1$. Then for all $t > 0$,
\[
\PP\left( \abs{P(X)} \leq t \right) \leq C m t^{1/m},
\]
where $C > 0$ is a universal constant.
\end{theorem}

\subsection{Langevin semigroup}

We now briefly recall some basic facts about the semigroup associated with a log-concave measure. For a detailed exposition, we refer to \cite{bakry2013analysis}. Let $\mu$ be a log-concave probability measure on $\RR^d$ with density
\[
\mu(dx) = e^{-V(x)}\,dx
\]
for a convex $V: \RR^d \rightarrow \RR$. The probability measure $\mu$
is associated with the symmetric diffusion operator
\begin{equation}
    L := \Delta - \nabla V \cdot \nabla,
\end{equation}
which satisfies, for smooth, compactly-supported functions $f,g: \RR^d \rightarrow \RR$,
\begin{equation}\label{eq597}
    \int f\, Lg \, d\mu = \int g\, Lf \, d\mu = -\int \nabla f \cdot \nabla g \, d\mu.
\end{equation}
Consider the Friedrich self-adjoint extension of the operator $L$ to a self-adjoint operator on $L^2(\mu)$,
which is also denoted by $L$. The corresponding semigroup is
\begin{equation}
    T_t := e^{tL} \qquad \qquad \qquad (t \geq 0).
\end{equation}
This semigroup is Markovian and admits an explicit probabilistic representation: if $(X_t)_{t \geq 0}$ solves the stochastic differential equation
\[
dX_t = \sqrt{2}\, dB_t - \nabla V(X_t)\, dt,
\]
where $(B_t)$ is standard Brownian motion in $\RR^d$, then $(X_t)$ is a Markov process, and
\[
T_t f(x) = \EE [f(X_t) \mid X_0 = x].
\]
It follows that $T_t$ is a contraction on $L^p(\mu)$ for all $p \geq 1$.
The semigroup $T_t$ has been widely used to establish functional inequalities for $\mu$, since it continuously interpolates between $T_0 f = f$ and $T_\infty f = \int f\, d\mu$. The rate at which $T_t f$ converges to the mean is governed by the gradient of $f$:

\begin{lemma}\label{lem309}
Let $f \in L^2(\mu)$ be a smooth function with $\int_{\RR^d} |\nabla f|^2 d \mu < \infty$. Then
\[
\norm{T_t f}_{L^2(\mu)}^2 \geq \norm{f}_{L^2(\mu)}^2 - 2t \int \abs{\nabla f}^2 \, d\mu.
\]
\end{lemma}

\begin{proof} The argument is standard, we sketch the computation:
\begin{align*}
\frac{d}{dt} \norm{T_t f}_{L^2(\mu)}^2 &= 2 \langle L T_t f,\, T_t f \rangle_{L^2(\mu)} \\
&= -2 \int \abs{\nabla T_t f}^2 \, d\mu \\
&\geq -2 \int T_t \abs{\nabla f}^2 \, d\mu \\
&\geq -2 \int \abs{\nabla f}^2 \, d\mu,
\end{align*}
where we used the standard gradient bound that follows from log-concavity
\[
\abs{\nabla T_t f}^2 \leq T_t \abs{\nabla f}^2,
\]
and the fact that $T_t$ is a contraction.
\end{proof}

Since $T_t$ acts as a local averaging operator, one may expect smoothing properties. It is well-known e.g. \cite{KL_bulletin} that $T_t$ maps bounded functions to Lipschitz functions:

\begin{lemma} \label{lem1162}
For every bounded function $f$ and any $t > 0$, we have
\[
\norm{T_t f}_{\mathrm{Lip}} \leq \frac{\norm{f}_\infty}{\sqrt{t}}.
\]
\end{lemma}

\subsection{Polynomial approximation of Lipschitz functions}

\subsubsection{Dimension 1}

For a log-concave measure $\mu$ on $\RR$, we define $\Psi_\mu(m)$ as the best function such that, for any $1$-Lipschitz function $f$ and integer $m \geq 1$,
\begin{equation}\label{eq653}
    E_m(\mu, f) := \inf_{\deg(P_m) \leq m} \norm{f - P_m}_{L^2(\mu)} \leq \Psi_\mu(m).
\end{equation}

We begin with a well-known qualitative result.

\begin{proposition}
Let $\mu$ be a measure on $\RR^d$ such that there exists $\varepsilon > 0$ such that for all $\theta \in B_2(0, \varepsilon)$,
\[
\int e^{\theta \cdot x} \, \mu(dx) < \infty.
\]
 Then polynomials are dense in $L^2(\mu)$. Moreover, the convergence is uniform over the class $\cF_{\mathrm{Lip}}$ of $1$-Lipschitz functions:
\[
E_m(\mu, \cF) := \sup_{f \in \cF} E_m(f, \mu) \longrightarrow 0 \quad \text{as } m \to \infty.
\]
\end{proposition}

\begin{proof}
Let $f \in L^2(\mu)$ be orthogonal to all polynomials, and define the signed measure $\mu_f(dx) = f(x)\mu(dx)$. By the Cauchy–Schwarz inequality, for $\theta$ small enough:
\[
\left( \int e^{\theta \cdot x} \, \mu_f(dx) \right)^2
\leq \left( \int f^2 \, d\mu \right) \left( \int e^{2\theta \cdot x} \, \mu(dx) \right) < \infty.
\]
Thus, $\mu_f$ admits a Laplace transform defined in a neighborhood of $0$, and all of its derivatives at the origin vanish due to the orthogonality condition. It follows that $\mu_f = 0$, hence $f = 0$ in $L^2(\mu)$. The uniform convergence follows by compactness of the set of centered $1$-Lipschitz functions in $L^2(\mu)$.
\end{proof}

In particular, since any log-concave probability measure satisfies the exponential integrability condition, we have
\[
\Psi_\mu(m) \longrightarrow 0
\]
as $m \to \infty$.

\medskip
We now turn to quantitative statements. A classical result in this direction is the quantitative Weierstrass approximation theorem, going back to Bernstein and to Jackson. It asserts that for a $1$-Lipschitz function $f$ on the interval $[-1, 1]$,
\begin{equation}\label{eq_jackson}
    \inf_{\deg(P) \leq m} \norm{f - P}_{L^\infty([-1,1])} \leq \frac{C}{m}.
\end{equation}
Let us describe in some detail how to obtain an $L^2$ version of this result. Let $\mu$ be the uniform probability measure on $[-1,1]$. The orthogonal polynomials with respect to $\mu$ are the Legendre polynomials,
\[
P_n(x) = \frac{1}{2^n n!} \frac{d^n}{dx^n} \left( x^2 - 1 \right)^n,
\]
which satisfy
\[
\int_{-1}^1 P_n(x) P_m(x) \, d\mu(x) = \frac{1}{2n + 1} \delta_{nm}.
\]
They also satisfy the differential equation
\begin{equation}\label{eq465}
    \frac{d}{dx} \left( (1 - x^2) P_n'(x) \right) + n(n + 1) P_n(x) = 0.
\end{equation}
Integrating by parts gives
\begin{equation}\label{eq366}
    \int_{-1}^1 P_n'(x) P_m'(x) (1 - x^2) \, dx = \frac{2n(n+1)}{2n+1} \delta_{nm}.
\end{equation}
Let us normalize the Legendre polynomials as
\[
p_n := \frac{P_n}{\sqrt{2n + 1}},
\]
so that $(p_n)$ forms an orthonormal basis in $L^2(\mu)$. Any function $f \in L^2(\mu)$ can be expanded as
\[
f = \sum_{k \geq 0} \langle f, p_k \rangle_{L^2(\mu)} p_k = \sum_{k \geq 0} f_k p_k.
\]
Observe that if $f' \in L^2((1 - x^2)\mu)$, then
\[
f' = \sum_{k \geq 1} f_k p_k',
\]
and by orthogonality using \eqref{eq366}, we obtain
\begin{equation}\label{eq376}
    \int (f')^2(1 - x^2) \, d\mu = \sum_{k \geq 1} k(k + 1) f_k^2.
\end{equation}
In particular,
\begin{equation}
    E_m^2(\mu, f) = \sum_{k \geq m + 1} f_k^2 \leq \frac{1}{(m + 1)(m + 2)} \int (f')^2(1 - x^2) \, d\mu.
\end{equation}
Following Jackson's theorem, an extensive body of work was devoted to the study of polynomial approximation on $\RR$ under more general probability measures (or "weights"). A good reference is the survey \cite{lubinsky2007survey}. Let us denote
\[
\mu_\alpha := \frac{1}{Z_\alpha} e^{-\abs{x}^\alpha},
\]
where $Z_\alpha$ is a normalizing constant. It can be shown that polynomials are dense in $L^2(\mu_\alpha)$ if and only if $\alpha \geq 1$. When $\alpha > 1$, it was proved by Freud \cite{freud1977markov} and Lubinsky and Levin \cite{levin1987canonical} that, for sufficiently regular functions $f$,
\begin{equation}\label{eq26}
    E_m^2(f, \mu_\alpha) \lesssim m^{2 - 2/\alpha} \int_\RR (f')^2 \, d\mu_\alpha.
\end{equation}
The case $\alpha = 1$ is different: it can be shown that the set
\[
\left\{ f \in L^2(\mu_1) : \int f \, d\mu_1 = 0, \quad \int (f')^2 \, d\mu_1 \leq 1 \right\}
\]
is not compact in $L^2(\mu_1)$. We refer to \cite{bakry2013analysis} for details. As a consequence, a bound of the form \eqref{eq26} cannot hold with any fixed rate. Nevertheless, a corollary of a result by \cite{lubinsky2007survey} shows that if $f$ is $1$-Lipschitz, then
\begin{equation}\label{eqLubinsky}
    E_m^2(f, \mu_1) \lesssim \frac{1}{\log^2(m + 1)}.
\end{equation}
As a consequence, we may state a universal approximation rate for log-concave probability measures on the real line:

\begin{lemma}
Let $\mu$ be a log-concave probability measure on $\RR$ with unit variance. Then for any $1$-Lipschitz function $f$ and $m \geq 1$,
\[
E_m^2(\mu, f) \lesssim \frac{1}{\log^2(1 + m)}.
\]
\end{lemma}

\begin{proof}
This follows from the fact that if $\rho$ is a log-concave density on $\RR$ with unit variance and barycenter at $0$, then
\[
\rho(x) \leq C \, e^{-\abs{x} / C}
\]
for some universal constant $C > 0$. A proof of this estimate can be found in \cite{bobkov2003concentration}.
\end{proof}

We also mention that, in sharp contrast with the two-sided exponential distribution, the one-sided exponential enjoys a much faster approximation rate. Denote by $\nu = e^{-x} \, \ind_{\RR^+}(x)\, dx$. Then for $1$-Lipschitz functions $f$,
\[
E_m^2(\nu, f) \lesssim \frac{1}{m},
\]
see e.g., \cite{bizeul2025polynomial}.

\subsubsection{Higher dimensions}

In higher dimensions, quantitative results on polynomial approximation are scarce. A notable exception is the case of the Gaussian measure $\gamma_d$ on $\RR^d$, which admits an explicit orthogonal basis of polynomials: the Hermite polynomials.
In dimension one, the $n$-th Hermite polynomial is defined via the Rodrigues formula:
\begin{equation}
    H_n(x) := (-1)^n e^{x^2/2} \frac{d^n}{dx^n} e^{-x^2/2},
\end{equation}
and satisfies the orthogonality relation:
\begin{equation}
    \int_{-\infty}^{+\infty} H_n(x) H_m(x) \, e^{-x^2/2} \, dx = n! \sqrt{2\pi} \, \delta_{n,m}.
\end{equation}
In dimension $d$, for a multi-index $\alpha = (\alpha_1, \dots, \alpha_d)$, define
\[
H_\alpha(x) := H_{\alpha_1}(x_1) \cdots H_{\alpha_d}(x_d).
\]
The Hermite polynomials form an orthogonal basis of $L^2(\gamma_d)$ and are also eigenfunctions of the differential operator
\[
L := \Delta + x \cdot \nabla,
\]
which is the infinitesimal generator of the Ornstein–Uhlenbeck semigroup:
\[
T_t f(x) := \EE \left[ f\left( \delta_t x + \sqrt{1 - \delta_t^2} G \right) \right],
\]
where $G \sim \gamma_d$ and $\delta_t = e^{-t}$.

\begin{proposition}\label{prop397}
For any multi-index $\alpha$, we have:
\[
L H_\alpha = -\abs{\alpha} \, H_\alpha,
\]
and consequently,
\[
T_t H_\alpha = e^{-t \abs{\alpha}} H_\alpha.
\]
\end{proposition}

Given a function $f \in L^2(\gamma_d)$, we may decompose it in the Hermite basis as
\[
f = \sum_\alpha f_\alpha H_\alpha.
\]
Using the integration by parts formula \eqref{eq597}, we compute the gradient energy:
\begin{align*}
\int \norm{\nabla f}^2 \, d\gamma_d
&= \int (-Lf) f \, d\gamma_d \\
&= \sum_{\alpha} \abs{\alpha} f_\alpha^2.
\end{align*}
Finally,
\begin{align*}
    E_m^2(\gamma_d, f) = \sum_{\abs{\alpha} \geq m+1} f_\alpha^2
    &\leq \frac{1}{m+1} \sum_{\alpha} \abs{\alpha} f_\alpha^2 \\
    &\leq \frac{1}{m+1} \int \abs{\nabla f}^2 \, d\gamma_d.
\end{align*}

This is an instance of the tensorization principle established in \cite{bizeul2025polynomial}.

\begin{proposition}[Tensorization]\label{prop_tensorization}
Let $\mu_i$ be a probability measure on $\RR$ for $i = 1, \ldots, d$. Assume that for all $i$,
\[
\sum_{\alpha \geq 1} \varphi_i(\alpha) f_\alpha^2 \leq \int_\RR (f')^2 w_i(x) \, d\mu_i,
\]
for some positive functions $(\varphi_i)_{1 \leq i \leq d}$ and $(w_i)_{1 \leq i \leq d}$, where $f_\alpha$ denotes the coefficients in the orthonormal polynomial basis of $\mu_i$.
Let $\mu := \mu_1 \otimes \cdots \otimes \mu_d$. Then for all smooth $f \in L^2(\mu)$,
\[
\sum_{\abs{\alpha} \geq 1} \varphi(\alpha) f_\alpha^2 \leq \int_{\RR^d} \sum_{i=1}^d w_i(x_i) (\partial_i f)^2 \, d\mu,
\]
where $\varphi(\alpha) := \sum_i \varphi_i(\alpha_i)$ and we set $\varphi_i(0) := 0$, $\abs{\alpha} := \sum_i \alpha_i$.
In particular, defining
\[
\Phi(m) := \inf_{\abs{\alpha} = m} \varphi(\alpha),
\]
we obtain the approximation bound
\[
E_m^2(\mu, f) \leq \frac{1}{\Phi(m+1)} \int_{\RR^d} \sum_{i=1}^d w_i(x_i) (\partial_i f)^2 \, d\mu.
\]
Moreover, if $f$ is $1$-Lipschitz,
\[
E_m^2(\mu, f) \leq \frac{1}{\Phi(m+1)} \, \EE \left[ \max_i w_i(X_i) \right].
\]
\end{proposition}

Let us illustrate Proposition~\ref{prop_tensorization} in concrete examples. For $1<\beta\leq2$, define the product measure
\[
\mu_\beta^{\otimes d} := \mu_\beta \otimes \cdots \otimes \mu_\beta.
\]
Recall that in dimension $1$, we have
\[
E_m^2(\mu_\beta, f) \lesssim m^{2/\beta - 2} \int_\RR (f')^2 \, d\mu_\beta,
\]
i.e., the tail bound
\[
\sum_{k \geq m+1} f_k^2 \lesssim m^{2/\beta - 2} \int_\RR (f')^2 \, d\mu_\beta.
\]
Using summation by parts, we deduce that for any $0 < \delta \leq 1$,
\[
\sum_{k \geq 1} \frac{k^{2-2/\beta} f_k^2}{\log^{1+\delta}(1 + k)} \leq \frac{C}{\delta} \int_\RR (f')^2 \, d\mu_\beta,
\]
for some constant $C > 0$. Define
\[
\varphi(x) := \frac{x^{2- 2/\beta }}{\log^{1+\delta}(x)}.
\]
Since $\varphi(x)/x$ decreases on $(1, \infty)$, we obtain
\[
\Phi(m) := \inf_{\abs{\alpha} = m} \sum_i \varphi(\alpha_i) = \varphi(m).
\]
By tensorization, this yields
\[
E_m^2(\mu_\beta^{\otimes d}, f) \lesssim \frac{\log^{1+\delta}(m)}{\delta \, m^{2-2/\beta}} \int_{\RR^d} \norm{\nabla f}^2 \, d\mu_\beta^{\otimes d}.
\]
Choosing
\[
\delta^* := \max\left(1, \frac{1}{\log \log m} \right),
\]
we obtain
\begin{equation}
E_m^2(\mu_\beta^{\otimes d}, f) \lesssim \frac{\log m \log \log m}{m^{2-2/\beta}} \int_{\RR^d} \norm{\nabla f}^2 \, d\mu_\beta^{\otimes d}.
\end{equation}
For the case $\beta = 1$, it was proved in \cite{bizeul2025polynomial} that
\begin{equation}\label{eq60}
\sum_{k = 1}^{\infty} \log^2(e + k) f_k^2 \lesssim \int_\RR \log^2(e + \abs{x}) (f')^2 \, d\mu_1.
\end{equation}
Hence, tensorization gives
\[
E_m^2(f, \mu_1^{\otimes d}) \lesssim \frac{1}{\log^2(1 + m)} \int_{\RR^d} \sum_{i = 1}^d \log^2(e + \abs{x_i}) (\partial_i f)^2 \, d\mu_1^{\otimes d}.
\]
If $f$ is $1$-Lipschitz, then
\[
E_m^2(f, \mu_1^{\otimes d}) \lesssim \frac{\EE \left[ \max_i \log^2(e + \abs{X_i}) \right]}{\log^2(m)} \lesssim \frac{\log^2 \log d}{\log^2 m}.
\]
In contrast, for the product of one-sided exponential measures, the approximation rate is much better. Let $\nu^{\otimes d}$ be the $d$-fold product of the one-sided exponential distribution. Then for $1$-Lipschitz~$f$,
\[
E_m^2(\nu^{\otimes d}, f) \lesssim \frac{\log d}{m},
\]
see \cite{bizeul2025polynomial}.

\medskip We conclude this section with an interesting dimensional effect of tensorization, which lies at the core of the entropy estimates of Section~\ref{sec_entropy}. Let $\mu$ be the uniform probability measure on $[-1, 1]$. Recall that for sufficiently regular $f$, using \eqref{eq376},
\[
\sum_{k \geq 1} k^2 f_k^2 \leq \int (f')^2 (1 - x^2) \, d\mu \leq \int (f')^2 \, d\mu.
\]
The rate is quadratic, much faster than the linear rate observed for the Gaussian measure, for example. For the uniform measure on the hypercube $\mu^{\otimes d}$, the tensorization principle yields, in particular,
\[
E_m^2(\mu, f) \leq \frac{1}{\Phi(m+1)} \int \norm{\nabla f}^2 \, d\mu^{\otimes d},
\]
where
\[
\Phi(m) = \inf_{\abs{\alpha} = m} \sum_i \alpha_i^2.
\]
The key difference is that here the function $\varphi(x) = x^2$ is convex, so that $\varphi(x)/x$ is increasing (as opposed to decreasing). Therefore,
\[
\Phi(m) = m, \qqquad \text{for } m \leq d,
\]
while for $m \geq d$,
\[
\Phi(m) \simeq \frac{m^2}{d}.
\]
The takeaway is that when the degree $m$ is smaller than or comparable to the dimension $d$, the rate cannot be better than the Gaussian one, i.e., linear in $m$.

\medskip
More precisely, let $\mu$ be an isotropic product measure on $\RR^d$, and let $\Phi_\mu$ denote the best function such that for all sufficiently regular functions $f$,
\[
E_m^2(\mu, f) \leq \frac{1}{\Phi_\mu(m)}.
\]
Then, for $m \leq d$,
\[
\Phi_\mu(m) \leq \Phi_\gamma(m) = m+1.
\]
The corresponding extremal function $f$ is the multilinear polynomial of degree $m \leq d$ defined by
\[
f(x) = \prod_{i = 1}^{m} x_i.
\]
Whenever the measure $\mu$ is isotropic and of product form, the function $f$ belongs to the tensor basis of orthonormal polynomials. Therefore, for all $k < m$,
\[
E_k^2(\mu, f) = \norm{f - 0}_{L^2(\mu)}^2 = 1.
\]
On the other hand, a direct computation shows that
\[
\int \norm{\nabla f}^2 \, d\mu = m.
\]
Hence,
\[
\Phi_\mu(m - 1) \leq \frac{\int \norm{\nabla f}^2 \, d\mu}{E_{m - 1}^2(\mu, f)} \leq m.
\]
We will see in Section~\ref{sec_entropy} that if $\mu$ is additionally log-concave, this remains true when restricting to Lipschitz functions,
at least for $m \leq \sqrt{d}$.

\section{Empirical computation of an approximating polynomial}\label{sec_algo}

In this section we analyze in detail the empirical procedures introduced in the introduction.
Given observations \eqref{eq138}, our goal is to construct a polynomial estimator of $f$,
taking advantage of the approximation property \eqref{eq177}.
We focus on two natural algorithms: the projection estimator, which relies on an orthogonal polynomial basis of $\mu$,
and the least-squares estimator, which requires no structural knowledge of $\mu$.

\subsection{The projection estimator}\label{subsec_projection}

We fix a $1$-Lipschitz function $f$, and assume that $\mu$ is a log-concave probability measure on $\RR^d$ with polynomial approximation rate:
\[
\sup_{f} \inf_{\deg(P) \leq m} \norm{f - P}_{L^2(\mu)} \leq \Psi_{\mu}(m),
\]
where the $\sup$ runs over all $1$-Lipschitz functions $f$. Recall that for normalization purposes, we assume that $\Psi_{\mu}(0) = 1$, which amounts
to the bound $C_P(\mu) \lesssim 1$.

\medskip We decompose the function $f$  as
\begin{equation}
    f = f - \EE_\mu f + \EE_\mu f =: \tilde{f} + a,
\end{equation}
where $a = \EE_\mu f$ is a constant and $\tilde{f}$ is mean-zero.

\medskip
Recall that we denote by $X_1, \dots, X_n$ the observed i.i.d. samples from $\mu$. The integer $n$ denotes the sample size used in the algorithm, while $m$ denotes the maximal polynomial degree used. Finally, we define
\[
D = D(d, m) := \binom{d + m}{m},
\]
which is the dimension of the space $\cP_{d,m}$ of multivariate polynomials of total degree at most $m$ in $\RR^d$.
Let us further denote by $(p_k)_{0 \leq k \leq D-1}$ an orthonormal basis of $\cP_{d,m} \subseteq L^2(\mu)$. Thus, we may decompose $f$ as
\begin{equation}\label{eq458}
    f = \sum_{k = 0}^{D-1} f_k \, p_k + f_{>m} = a + \sum_{k = 1}^{D - 1} f_k \, p_k + f_{>m},
\end{equation}
where for all $1 \leq k \leq D - 1$,
\[
f_k := \langle f, p_k \rangle_{L^2(\mu)} = \langle \tilde{f}, p_k \rangle_{L^2(\mu)},
\]
and
\[
\norm{f_{>m}}_{L^2(\mu)} \leq \Psi(m).
\]
Recall that the empirical estimator of the mean is given by
\[
\hat{a} := \hat{f}_0 = \frac{1}{n} \sum_{i = 1}^n Y_i,
\]
and for $1 \leq k \leq D - 1$, we define the empirical coefficients as
\begin{align*}
    \hat{f}_k &= \frac{1}{n} \sum_{i = 1}^n (Y_i - \hat{a}) \, p_k(X_i) \\
    &= \frac{1}{n} \sum_{i = 1}^n (Y_i - a) \, p_k(X_i) + \frac{1}{n} \sum_{i = 1}^n (\hat{a} - a) \, p_k(X_i) \\
    &=: \hat{f}_k^* + \delta_k. \numberthis \label{eq491}
\end{align*}
Here, $\hat{f}_k^*$ is the unbiased component of the estimator, satisfying
\[
\EE[\hat{f}_k^*] = f_k.
\]
The projection estimator is then given by
\[
\hat{f} := \hat{a} + \sum_{k = 1}^{D - 1} \hat{f}_k \, p_k.
\]
Our goal is to prove the following result, previously stated in the introduction.

\begin{theorem}\label{thm:avg_error_fourier}
Under the above assumptions, the projection estimator satisfies
\begin{equation}
    \EE \norm{f - \hat{f}}_{L^2(\mu)}^2 \leq \Psi_\mu^2(m) + \frac{(C m^2+4\sigma^2) D}{n} ,
\end{equation}
for some absolute constant $C > 0$.

Furthermore, in the Gaussian setting, we have the sharper bound
\begin{equation}\label{eq506}
    \EE \norm{f - \hat{f}}_{L^2(\gamma)}^2 \leq \frac{1}{m} + \frac{(C m+4\sigma^2) D}{n},
\end{equation}
for some absolute constant $C > 0$.
\end{theorem}
\begin{proof}
Let $P_m$ denote the orthogonal projection of $f$ onto the space $\cP_{d,m}$ of polynomials of degree at most $m$. Then:
\begin{align*}
    \EE \norm{f - \hat{f}}_2^2
    &= \EE \norm{f - P_m}_2^2 + \EE \norm{\hat{f} - P_m}_2^2 \\
    &\leq \Psi^2(m) + \Var(\hat{a}) + \sum_{k = 1}^{D - 1} \EE (\hat{f}_k - f_k)^2 \\
    &\leq \Psi^2(m) + \Var(\hat{a}) + 2 \sum_{k = 1}^{D - 1} \EE (\hat{f}_k^* - f_k)^2 + 2 \sum_{k = 1}^{D - 1} \EE \delta_k^2 \\
    &= \Psi^2(m) + \frac{\Var(Y_1)}{n} + 2 \sum_{k = 1}^{D - 1} \Var(\hat{f}_k^*) + 2 \sum_{k = 1}^{D - 1} \EE \delta_k^2 \\
    &\leq \Psi^2(m) + \frac{1 + \sigma^2}{n} + \frac{2}{n} \sum_{k = 1}^{D - 1} \Var\left((Y_1 - a) p_k(X_1)\right) + 2 \sum_{k = 1}^{D - 1} \EE \delta_k^2, \numberthis \label{eq513}
\end{align*}
where in the last passage we used that $Var_{\mu}(f) \leq \Psi_{\mu}(0) = 1$.
We first bound the third term in (\ref{eq513}). Let $(X, Y)$ denote a copy of $(X_1, Y_1)$. Observe that
\[
Y - a = f(X) + \xi - a = \tilde{f}(X) + \xi,
\]
where $\tilde{f} = f - a$ is centered. Then, for any $1 \leq k \leq D$,
\begin{align*}
    \Var\left((Y - a)p_k(X)\right)
    &\leq \EE \left( (Y - a)^2 p_k^2(X) \right) \\
    &= \EE \left( \tilde{f}^2(X) p_k^2(X) \right) + \EE \left( \xi^2 p_k^2(X) \right) \\
    &= \EE \left( \tilde{f}^2(X) p_k^2(X) \right) + \sigma^2.
\end{align*}

Now we apply Hölder's inequality with exponents $q = m + 1$ and $q^* = 1 + 1/m$:
\begin{align*}
    \EE \left[ \tilde{f}^2(X) p_k^2(X) \right]
    &\leq \left( \EE \tilde{f}^{2q}(X) \right)^{1/q} \left( \EE p_k^{2q^*}(X) \right)^{1/q^*} \\
    &\leq C \norm{\tilde{f}(X)}_{2m+2}^2,
\end{align*}
as follows from Proposition \ref{prop:poly_concentration_logconcave}. Recalling
that $C_P(\mu) \lesssim 1$, by Proposition~\ref{prop:lipschitz_concentration}, we have
\begin{equation}\label{eq884}
    \norm{\tilde{f}(X)}_{2m+2}^2 \lesssim m^2,
\end{equation}
since $\tilde{f}$ is $1$-Lipschitz and centered.
Now we bound the fourth term in~\eqref{eq513}. Define
\[
S_n := \frac{1}{n} \sum_{i = 1}^n p_k(X_i).
\]
Then:
\begin{align*}
    \EE \delta_k^2
    &= \EE \left( \hat{a} - a \right)^2 S_n^2 \\
    &= \EE \left( \frac{1}{n} \sum_{i = 1}^n \left( \tilde{f}(X_i) + \xi_i \right) \right)^2 S_n^2 \\
    &= \EE \left( \frac{1}{n} \sum_{i = 1}^n \tilde{f}(X_i) \right)^2 S_n^2 + \EE \left( \frac{1}{n} \sum_{i = 1}^n \xi_i \right)^2 S_n^2 \\
    &= \EE \left( \frac{1}{n} \sum_{i = 1}^n \tilde{f}(X_i) \right)^2 S_n^2 + \frac{\sigma^2}{n^2} \EE S_n^2.
\end{align*}
Note that $\EE S_n^2 = \EE P_k^2(X) / n = 1/n$.
By again using H\"older's inequality  with $q = m+1$, and bounding $\norm{S_n}_{q^*} \leq \norm{p_k(X)}_{q^*}$:
\begin{align*}
    \EE \delta_k^2
    &\leq \norm{\frac{1}{n} \sum_{i = 1}^n \tilde{f}(X_i)}_{q}^2 \cdot \norm{p_k(X)}_{q^*}^2 + \frac{\sigma^2}{n^3} \\
    &\lesssim \norm{\frac{1}{n} \sum_{i = 1}^n \tilde{f}(X_i)}_{m+1}^2 + \frac{\sigma^2}{n^2}.
\end{align*}
By Proposition~\ref{prop:bernstein}, we know that for a $1$-Lipschitz function:
$$
\norm{\frac{1}{n} \sum_{i = 1}^n \tilde{f}(X_i)}_{\psi_1} \lesssim \frac{1}{\sqrt{n}} \norm{\tilde{f}(X)}_{\psi_1} \lesssim \frac{1}{\sqrt{n}}.
$$
Hence,
\begin{equation}
\norm{\frac{1}{n} \sum_{i = 1}^n \tilde{f}(X_i)}_{m+1}^2 \lesssim \frac{m^2}{n}. \label{eq_534} \end{equation}
Plugging everything back into~\eqref{eq513}, we obtain:
\[
\EE \norm{f - \hat{f}}_2^2 \leq \Psi_\mu^2(m) + \frac{C m^2 D}{n} + \frac{4 \sigma^2D}{n},
\]
for some absolute constant $C > 0$, as claimed. For the ``Furthermore'' part,
we replace $m^2$ in (\ref{eq884}) and in (\ref{eq_534}) by
$m \cdot \min \{m, \rho_{LS}(X) \}$ which equals $m$ in the Gaussian case.

\subsubsection{Proof of Theorem \ref{thm_proj}}\label{subsub_proj}

Let us explain how to deduce Theorem \ref{thm_proj} from Theorem \ref{thm:avg_error_fourier}. Set
$$5\leq m_0 = \left \lfloor\frac{\log n}{\log d} \right \rfloor \leq \frac{d}{2}$$
and observe that
\begin{align*}
    \binom{d+m_0}{m_0} &\leq \left(\frac{e(d+m_0)}{m_0}\right)^{m_0}\\
    &\leq \left(\frac{5d}{m_0}\right)^{m_0}\\
    &\leq d^{m_0} \leq n.
\end{align*}
For any integer $1\leq p \leq m_0$ we set
$$m_p = m_0-p.$$
Recall that $D = D(d, m_p) = {d + m_p \choose m_p}$.
According to the preceding inequality, we have :
\begin{align*}
    \frac{\binom{d+m_p}{m_p}}{n} &\leq \frac{\binom{d+m_0-p}{m_0-p}}{\binom{d+m_0}{m_0}}
    = \frac{(d + m_0-p)!}{(d+m_0)!} \cdot \frac{ m_0!}{(m_0-p)! } \\ & \leq \left(\frac{m_0}{d}\right)^p.
\end{align*}
Plugging this into Theorem \ref{thm:avg_error_fourier}, we get for the choice of $m=m_p$,
$$\EE\normmu{f-\hat{f}}^2 \leq \Psi^2_\mu(m_p) + \left(Cm_0^2+4\sigma^2\right)\left(\frac{m_0}{d}\right)^p.$$
We choose
$$p = \max \left(4,\left\lceil\frac{4\log m_0}{\log d/m_0}\right\rceil \right).$$
In the first regime,
$$n\leq e^{\sqrt{d}\log d}$$
and $p=4$ while $m_0 \leq \sqrt{d}$. We have
$$\EE\normmu{f-\hat{f}}^2 \leq \Psi^2_\mu(m_0-4) + C'd(1/\sqrt{d})^4$$
which is what we wanted to proved. In the second regime, we have
$$e^{\sqrt{d}\log d}\leq n\leq e^{d\log d/2},$$
$$p = \left \lceil \frac{4\log(m_0)}{\log(d/m_0)} \right \rceil$$
and
$$\sqrt{d} - 1 \leq m_0 \leq d/2.$$
By our choice of $p$,
$$\left(\frac{m_0}{d}\right)^p \leq \left(\frac{m_0}{d}\right)^{\frac{4\log(m_0)}{\log(d/m_0)}} = \frac{1}{m_0^4}$$
which concludes the proof.
\end{proof}

As can be seen from the proof of Theorem \ref{thm_proj}, the error term in
(\ref{eq259}) may be improved to $O(1/d^5)$ or so if we take $m = m_0 12$ rather than $m = m_0 -4$. In any case, the error term is typically negligible compared to $\Psi_{\mu}^2(m)$.

\subsection{Least square estimator}
We now move to the analysis of the least squares estimator $\hat{f}_{LS}$. Given a choice of a polynomial degree $m$, this estimator  is defined as the polynomial of degree less than $m$ that minimizes the empirical $l^2$ risk:
\begin{equation}\label{eq565}
    \hat{f}_{LS} = \argmin_{P \in \cP_{d,m}} \sum_{i=1}^n\left(Y_i - P(X_i)\right)^2.
\end{equation}
The goal of this section is to prove the following bound on its prediction error:
\begin{theorem}\label{thm:avg_ls}
    Assume that
    $$ \sigma^2\leq d$$
    and that
    $d\geq d_0,$     for some universal constant $d_0\geq3$. Then for any $n, m \geq 1$ such that the right-hand side is smaller than $1$, it holds that
    \begin{equation}\label{eq643}
        \EE\normmu{f - \hat{f}_{LS}}^2 \leq \Psi^2_\mu(m) + \frac{(C\log(n))^{2m} D \log(D)}{n} + \frac{8\sigma^2D}{n},
    \end{equation}
     for some absolute constant $C \geq 1$.
\end{theorem}

Before embarking on the proof of Theorem \ref{thm:avg_ls}, we remark that the assumptions implies in particular that
    \begin{equation}\label{eq1196}
        \log(n)^{m} D \leq (C\log(n))^{2m}D \leq n.
    \end{equation}
    Thus, if $n\geq 3$, we get
    \begin{equation}\label{eq1200}
        m\leq \log n.
    \end{equation}
    Furthermore,
    \begin{equation}\label{eq1204}
        D = \binom{d+m}{m} \geq \left(\frac{d}{m}\right)^m.
    \end{equation}
    Plugging \eqref{eq1204} and \eqref{eq1200} into \eqref{eq1196} we get
    $$d^m\leq n,$$
    that is
\begin{equation}\label{eq1210}
        m\leq \frac{\log n}{\log d}.
\end{equation}
which we assume in the rest of this section.
Note that the quantity of interest,
$$f - \hat{f}_{LS},$$
is unchanged if we subtract a constant from $f$. Thus, for convenience, we assume from now on in this section that
\begin{equation}
    \int f \, d\mu = 0.
\end{equation}
We define
$$A = (p_k(X_i))_{k,i} = \begin{pmatrix}
    p_0(X_1) & \dots & p_{D-1}(X_1) \\
    \vdots & & \vdots \\
    p_0(X_n) & \dots & p_{D-1}(X_n)
\end{pmatrix} \in \RR^{n \times D},$$
and
$$b = \begin{pmatrix}
    Y_1 \\
    \vdots \\
    Y_n
\end{pmatrix}.$$

We adopt the following useful notation: for a polynomial $P$ of degree at most $m$, we write
$$\boldsymbol{P} \in \RR^D$$
for the vector of its coordinates in the basis $(p_k)_{0 \leq k \leq D-1}$. A straightforward computation shows that
\begin{align}
    \boldsymbol{\hat{f}_{LS}} = (A^T A)^{-1} A^T b &= \left(\frac{1}{n} A^T A\right)^{-1} \frac{1}{n} A^T b. \label{eq580}
\end{align}

The vector $\frac{1}{n} A^T b$ is merely the vector of empirical scalar products, which, as in the previous section, we denote by $\boldsymbol{\hat{f}^*}$. For all $0 \leq k \leq D - 1$,
\begin{equation}\label{eq599}
    \boldsymbol{\hat{f}^*_k} = \left(\frac{1}{n} A^T b\right)_k = \frac{1}{n} \sum_{i=1}^n Y_i p_k(X_i).
\end{equation}
This is indeed the same definition as in (\ref{eq491}), since we assumed that
$$a = \int f \, d\mu = 0.$$
From the analysis carried out in Section \ref{subsec_projection}, we know that
\begin{equation}\label{eq1007}
\EE\norm{ \boldsymbol{P_mf}- \boldsymbol{\hat{f}^*}}^2_2 \leq \frac{(Cm^2+4\sigma^2)D}{n}
\end{equation}
where $P_mf$ is the projection of $f$ onto $\cP_{d,m}$ in $L^2(\mu)$. From now on, we assume that $n$ is large enough so that the right hand side in \eqref{eq643} is smaller than $1$. In particular, we also get
$$
    \EE\normmu{\boldsymbol{P_mf}-\hat{f}^*}^2 \lesssim 1. $$
    Since $\| f \|_{L^2(\mu)} \leq \Psi_{\mu}(0) = 1$,
\begin{equation}
\| \hat{f}^* \|_{L^2(\mu)} \leq  \| \boldsymbol{P_mf}-\hat{f}^*
\|_{L^2(\mu)} + \| \boldsymbol{P_mf} \|_{L^2(\mu)}
\leq C + \| f \|_{L^2(\mu)} \leq \tilde{C}.
\label{eq617}
\end{equation}
We denote
$$C_n = \frac{1}{n}A^TA.$$
The main technical step in this section is a moment bound on the deviation of $C_n^{-1}$ from the identity matrix, measured in operator norm.
\begin{lemma}\label{lem688}
Assume that $n\geq D$, then for all $1\leq p \leq \log D$,
\begin{equation*}
    \left(\EE\norm{C_n^{-1}-I_d}^p_{op}\right)^{2/p} \leq \frac{(C\log n)^{2m} D\log D}{n}
\end{equation*}
where $C>0$ is a universal constant.
\end{lemma}
Before proving this lemma, let us explain how it implies Theorem \ref{thm:avg_ls}. As a warm-up, we first prove a weaker statement:
\begin{equation}\label{eq571}
        \EE\norm{f - \hat{f}_{LS}}_2 \leq \Psi_\mu(m) + \sqrt{\frac{(C\log n)^{2m} D \log D}{n}} + \frac{2\sigma \sqrt{D}}{\sqrt{n}} .
\end{equation}

\subsubsection{Proof of \eqref{eq571}}
Since $\boldsymbol{\hat f^{*}}$ is given by \eqref{eq599}, we may write
\begin{align*}
    \EE\norm{\boldsymbol{P_m f}-\boldsymbol{\hat f_{LS}}}_2
    &\leq \EE\norm{\boldsymbol{P_m f}-\boldsymbol{\hat f^{*}}}_2 + \EE\norm{\boldsymbol{\hat f_{LS}}-\boldsymbol{\hat f^{*}}}_2\\
    &= \EE\norm{\boldsymbol{P_m f}-\boldsymbol{\hat f^{*}}}_2 + \EE\norm{(C_n^{-1}-I_d)\,\boldsymbol{\hat f^{*}}}_2\\
    &\leq \EE\norm{\boldsymbol{P_m f}-\boldsymbol{\hat f^{*}}}_2 + \EE\norm{C_n^{-1}-I_d}_{op}\,\norm{\boldsymbol{\hat f^{*}}}_2 .
\end{align*}
Now, using the Cauchy--Schwarz inequality and \eqref{eq617}, we bound the last term by
\[
\EE\norm{C_n^{-1}-I_d}_{op}\,\norm{\boldsymbol{\hat f^{*}}}_2
\;\lesssim\; \left(\EE\norm{C_n^{-1}-I_d}_{op}^2\right)^{1/2}.
\]
From Lemma \ref{lem688} with $p=2$, we know that
\begin{equation}\label{eq:l2_bound_inv_covariance}
    \EE\norm{C_n^{-1}-I_d}_{op}^2 \lesssim \frac{(C\log n)^{2m}D\log D}{n}.
\end{equation}
Putting everything together and using \eqref{eq1007}, we get
\begin{align*}
    \EE\normmu{f-\hat f_{LS}}
    &\leq \normmu{f-P_m f}+\EE\normmu{P_m f-\hat f_{LS}}\\
    &\leq \Psi_\mu(m) + \sqrt{\frac{(Cm^2+4\sigma^2)D}{n}} + \sqrt{\frac{(C\log n)^{2m}D\log D}{n}} \\
    &\leq \Psi_\mu(m) + \sqrt{\frac{(C\log n)^{2m}D\log D}{n}} + \frac{2\sigma \sqrt{ D}}{\sqrt{n}} ,
\end{align*}
where the constant $C$ may change from line to line, and we used $\sqrt{a+b}\leq \sqrt{a}+\sqrt{b}$.

\subsubsection{Proof of Theorem \ref{thm:avg_ls}}

The proof of \eqref{eq643} follows the same strategy, with one additional computation.
As before, we write
\begin{align*}
    \EE\normmu{f-\hat f_{LS}}^2
    &= \EE\normmu{f-P_m f}^2+\EE\normmu{P_m f-\hat f_{LS}}^2\\
    &\leq \Psi^2_\mu(m) + 2\,\EE\norm{\boldsymbol{P_m f}-\boldsymbol{\hat f^{*}}}_{2}^2 + 2\,\EE\norm{\boldsymbol{\hat f_{LS}}-\boldsymbol{\hat f^{*}}}_{2}^2 \\
    &\leq \Psi^2_\mu(m) + \frac{(Cm^2+8\sigma^2) D}{n} + 2\,\EE\norm{C_n^{-1}-I_d}_{op}^2\,\norm{\boldsymbol{\hat f^{*}}}_2^2 .
\end{align*}
Using H\"older's inequality with $p=\tfrac12\log D$ and $q=p^{*}$, and using Lemma \ref{lem688}, we bound
\begin{align*}
    \EE\norm{C_n^{-1}-I_d}_{op}^2\,\norm{\boldsymbol{\hat f^{*}}}_2^2
    &\leq \left(\EE\norm{C_n^{-1}-I_d}_{op}^{2p}\right)^{1/p}\left(\EE\norm{\boldsymbol{\hat f^{*}}}_2^{2q}\right)^{1/q}\\
    &\leq \frac{(C\log n)^{2m} D\log D}{n}\;\left(\EE\norm{\boldsymbol{\hat f^{*}}}_2^{2q}\right)^{1/q}.
\end{align*}
It remains to prove that
\begin{equation}\label{eq1062}
    \left(\EE\norm{\boldsymbol{\hat f^{*}}}_2^{2q}\right)^{1/q}\lesssim 1.
\end{equation}
We use a simple interpolation argument. Recall that, by \eqref{eq617},
\begin{equation}\label{eq1066}
    \EE\norm{\boldsymbol{\hat f^{*}}}_2^{2} \lesssim 1.
\end{equation}
Recall that $\sigma^2\leq d \leq D$. We claim the following crude bound on the fourth moment:
\begin{equation}\label{eq1070}
    \EE\norm{\boldsymbol{\hat f^{*}}}_2^{4} \lesssim D^{3}.
\end{equation}
Indeed,
\begin{align*}
    \EE\norm{\boldsymbol{\hat f^{*}}}_2^{4}
    &= \EE \Bigl(\sum_{k=0}^{D-1} (\hat f_k^{*})^2\Bigr)^2
      \;\leq\; D\sum_{k=0}^{D-1}\EE\bigl(\hat f_k^{*})^4.
\end{align*}
For any $0\leq k \leq D-1$,
\begin{align*}
    \EE\bigl((\hat f_k^{*})^4\bigr)
    &= \EE\!\left(\frac{1}{n}\sum_{i=1}^n Y_i\,p_k(X_i)\right)^{\!4}
      \;\leq\; \EE\!\left(Y_1^4\,p_k(X_1)^4\right) \\
    &\leq \bigl(\EE Y_1^8\bigr)^{1/2}\,\bigl(\EE |p_k(X_1)|^8\bigr)^{1/2}
      \;\lesssim\; \sigma^4\,C^{m}
      \;\lesssim\; D^2,
\end{align*}
where we used Propositions \ref{prop:lipschitz_concentration} and \ref{prop:poly_concentration_logconcave} and the growth assumption on $\sigma$.
Now, for any nonnegative random variable $X$ and any $q\in[1,2]$, by interpolation,
\[
\bigl(\EE X^{q}\bigr)^{1/q} \;\leq\; (\EE X)^{\,2/q - 1}\,(\EE X^{2})^{\,1 - 1/q}.
\]
Applying this to $X=\norm{\boldsymbol{\hat f^{*}}}_2^{2}$ and using \eqref{eq1066} and \eqref{eq1070}, we get
\begin{align*}
     \left(\EE\norm{\boldsymbol{\hat f^{*}}}_2^{2q}\right)^{1/q}
     &\lesssim (D^3)^{\,1 - 1/q}
      \;=\; D^{\,3/p}
      \;=\; D^{\,6/\log D}
      \;\lesssim\; 1,
\end{align*}
which proves \eqref{eq1062}.
In order to complete the proof of Theorem \ref{thm:avg_ls}, it remains to prove Lemma \ref{lem688}.

\subsubsection{Bounding $C_n^{-1}$}
The proof of Lemma \ref{lem688} consists of two steps. First, we prove the same bound but for $C_n$ rather than for its inverse.
\begin{lemma}\label{lem724}
Assume that $n\geq D$, then for all $1\leq p \leq \log D$,
\begin{equation*}
    \left(\EE\norm{C_n-I_d}_{op}^p\right)^{2/p} \leq \frac{(C\log n)^{2m} D\log D}{n}.
\end{equation*}
where $C>0$ is a universal constant.
\end{lemma}

In order to prove Lemma \ref{lem724}, we first unpack the definition of $C_n$. For $1\leq i \leq n$, define i.i.d random vectors
\begin{equation}
    Z_i =
    \begin{pmatrix}
    p_1(X_i) \\
    \vdots \\
    p_{D-1}(X_i)
    \end{pmatrix}
    \in\RR^{D-1} \qquad \qquad \qquad (i=1,\ldots,n).
\end{equation}
Notice that we did not include the term $p_0(X_i)=1$ corresponding to the constant polynomial. We write $Z$ for a random vector with the same law as $Z_1$. Then $Z$ is isotropic:
\[
\EE Z = 0, \qquad \EE ZZ^\top = I_d.
\]
We denote the empirical covariance of $Z$ by
\begin{equation}\label{eq645}
\Cov_n = \frac{1}{n}\sum_{i=1}^n Z_iZ_i^\top.
\end{equation}
We also write $\Tilde{Z_i}$ for the full vector
\begin{equation}\label{eq649}
    \Tilde{Z_i} =
    \begin{pmatrix}
    1 \\
    Z_i
    \end{pmatrix} \in \RR^D.
\end{equation}
We may then express $C_n$ as
\begin{equation}\label{eq654}
\begin{aligned}
    C_n &= \frac{1}{n}AA^\top
    = \frac{1}{n} \sum_{i=1}^n \Tilde{Z_i}\Tilde{Z_i}^\top \\[0.7em]
    &=
    \renewcommand{\arraystretch}{1.3}
    \begin{pmatrix}
    1 & \frac{1}{n} \sum_{i=1}^n Z_i^\top \\
    \frac{1}{n} \sum_{i=1}^n Z_i & \Cov_n
    \end{pmatrix}
    \renewcommand{\arraystretch}{1}.
\end{aligned}
\end{equation}
From \eqref{eq654} and \eqref{eq649}, we easily deduce the following lemma.
\begin{lemma}\label{lem670}
    Let $C_n$ and $\Cov_n$ be defined by \eqref{eq654} and \eqref{eq645}, respectively. Then,
    \begin{equation}\label{eq672}
        \norm{C_n-I_d}_{op} \leq \Bigl\|\frac{1}{n}\sum_{i=1}^n Z_i\Bigr\| \;+\; \norm{\Cov_n - I_d}_{op}.
    \end{equation}
\end{lemma}

In what follows, we bound the $p$-th moment of the operator norm of $\Cov_n - I_d$.
We use Rudelson's lemma \cite{rudelson1999random}, relying on  the non-commutative Khintchine inequality of Lust-Picard and Pisier (see \cite[Theorem 9.8.2]{pisier2003introduction}). Inequality (3.4) in
\cite{rudelson1999random} reads as follows:

\begin{lemma}\label{lem:l2_rudelson}
    Let $x_1,\dots,x_n$ be vectors in $\RR^D$, and let $\epsilon_1,\dots,\epsilon_n$ be i.i.d. symmetric Bernoulli   variables. Then for any $p\leq\log D$,
    \[
    \Bigl(\EE\norm{\sum_i \epsilon_i x_i\otimes x_i}_{op}^p\Bigr)^{2/p}
    \;\leq\; C\log D \,\max_i \norm{x_i}^2 \,\norm{\sum_i x_i\otimes x_i}_{op}.
    \]
\end{lemma}

As in Rudelson's paper, the lemma is used to bound the deviation of the empirical covariance from its expectation.
\begin{corollary}\label{cor:l2_bound_general_covariance}
    Let $\Cov_n$ be defined by \eqref{eq645}. Whenever the right-hand side is smaller than $1$,
    \[
    \Bigl(\EE\norm{\Cov_n-I_d}^p\Bigr)^{2/p} \;\leq\; \frac{C\log D}{n}\,\Bigl(\EE\max_i |Z_i|^{2p}\Bigr)^{1/p}.
    \]
\end{corollary}
We need a standard symmetrization lemma.
\begin{lemma}
Let $(X_i)_{i\in I}$ be a finite sequence of independent random
vectors in some Banach space, and let $\varepsilon_i$ be independent symmetric Bernoulli random variables. Then, for any $p\geq 1$,
\[
\EE\norm{\sum_{i\in I}X_i - \EE X_i}^p \;\leq\; 2^p\,\EE\norm{\sum_{i\in I}\varepsilon_iX_i}^p.
\]
\end{lemma}
\begin{proof}
    We set
    \[
    \Tilde{X_i} = X_i-X_i',
    \]
    where $X_i'$ is an independent copy of $X_i$. By Jensen's inequality,
    \begin{align*}
        \EE\norm{\sum_{i\in I}X_i - \EE X_i}^p
        &\leq \EE\norm{\sum_{i\in I}\Tilde{X_i}}^p \\
        &= \EE\norm{\sum_{i\in I}\varepsilon_i\Tilde{X_i}}^p\\
        &\leq 2^{p-1}\EE\!\left( \norm{\sum_{i\in I}\varepsilon_iX_i }^p+\norm{\sum_{i\in I}\varepsilon_iX_i' }^p \right)\\
        &=2^p\,\EE\norm{\sum_{i\in I}\varepsilon_iX_i}^p.
    \end{align*}
\end{proof}

We can now prove Corollary \ref{cor:l2_bound_general_covariance}.
\begin{proof}[Proof of Corollary \ref{cor:l2_bound_general_covariance}]
    We set
    \[
    S_p = \EE\norm{\Cov_n-I_d}_{op}^p,
    \]
    the quantity of interest. The first step is to use the symmetrization lemma:
    \begin{align*}
        S_p &= \EE\norm{\frac{1}{n}\sum_{i=1}^n (Z_i\otimes Z_i-\EE Z_i\otimes Z_i)}_{op}^p\\
        &\leq \frac{2^p}{n^p}\EE\norm{\sum_{i=1}^n\varepsilon_iZ_i\otimes Z_i}_{op}^p.
    \end{align*}
    We then apply Rudelson's lemma, conditionally on the $Z_i$’s and then take expectation over the $Z_i$’s, to obtain
    \begin{align}
        S_p &\leq \frac{C^p}{n^p}\log(D)^{p/2}\,\EE\!\left(\max_i\norm{Z_i}^p\,\norm{\sum_{i}Z_i\otimes Z_i}_{op}^{p/2}\right) \notag\\
        &\leq \frac{C^p}{n^p}\log(D)^{p/2}\,\Bigl(\EE\max_i\norm{Z_i}^{2p}\Bigr)^{1/2}\,\Bigl(\EE\norm{\sum_{i}Z_i\otimes Z_i}_{op}^{p}\Bigr)^{1/2}, \label{eq809}
    \end{align}
    where we used Cauchy--Schwarz. Now, observe that
    \begin{align*}
        \EE\norm{\sum_i Z_i\otimes Z_i}_{op}^p
        &= n^p\,\EE\norm{I_d+\frac{1}{n}\sum_i (Z_i\otimes Z_i - I_d)}_{op}^p\\
        &\leq 2^{p-1}n^p\Bigl(1+\EE\norm{\frac{1}{n}\sum_i (Z_i\otimes Z_i - I_d)}_{op}^p\Bigr)\\
        &= 2^{p-1}n^p(1+S_p).
    \end{align*}
    Plugging this back into \eqref{eq809}, we find that
    \begin{equation}\label{eq818}
        S_p \leq \lambda(1+S_p)^{1/2},
    \end{equation}
    where
    \[
    \lambda = \frac{(C\log D)^{p/2}}{n^{p/2}}\,\Bigl(\EE\max_i |Z_i|^{2p}\Bigr)^{1/2}.
    \]
    Distinguishing the cases $S_p\leq 1$ and $S_p\geq 1$, one obtains
    \[
        S_p\leq 2\max(\lambda,\lambda^2).
    \]
    Thus, in particular, if $\lambda\leq 1$, we conclude that
    \[
        \left(\EE\norm{\Cov_n-I_d}_{op}^p\right)^{2/p}
        \leq \lambda^{2/p}
        = \frac{C\log D}{n}\,\Bigl(\EE\max_i |Z_i|^{2p}\Bigr)^{1/p},
    \]
    which is the desired bound.
\end{proof}

We now need an estimate on
\[
\left(\EE\max_i |Z_i|^{2p}\right)^{1/p}.
\]
The $\ell_\infty$ norm on $\RR^n$  is equivalent to the $\ell_q$ norm for $q=2\log n$, up to a universal constant. For this choice of $q$, notice that $2p/q\leq 1$. By Jensen's inequality,
\begin{align*}
    \left(\EE\max_i\norm{Z_i}^{2p}\right)^{1/p}
    &\lesssim \left(\EE\Bigl(\sum_{i=1}^n |Z_i|^{q}\Bigr)^{2p/q}\right)^{1/p} \\
    &\lesssim \bigl(n\,\EE |Z_1|^{q}\bigr)^{2/q} \\
    &\lesssim \bigl(\EE |Z_1|^{q}\bigr)^{2/q}.
\end{align*}
Finally, the random variable $Q=|Z_1|^2$ is a degree $2m$ polynomial in log-concave variables with
\[
\EE |Q| = \EE Q = \EE|Z_1|^2 = D-1.
\]
By Proposition \ref{prop:poly_concentration_logconcave}, we obtain
\begin{align*}
    \bigl(\EE|Z_1|^{q}\bigr)^{2/q}
    &= \bigl(\EE Q^{q/2}\bigr)^{2/q} \\
    &\leq (Cq/2)^{2m}\,D \\
    &= D(C\log n)^{2m}.
\end{align*}
At this point, we have established the bound (whenever the right-hand side is smaller than $1$):
\begin{equation}
    \left(\EE\norm{\Cov_n-I_d}_{op}^p\right)^{2/p} \leq \frac{(C\log n)^{2m} D\log D}{n}.
\end{equation}
In view of Lemma \ref{lem670}, we have, for any $p\geq 1$,
\begin{align}
    \left(\EE\norm{C_n-I_d}_{op}^p\right)^{2/p}
    &\leq 2\left(\EE\norm{\Cov_n-I_d}_{op}^p\right)^{2/p}
      + 2\EE \left| \frac{1}{n}\sum_{i=1}^n Z_i \right|_2^p \notag\\
    &\leq \frac{(C\log n)^{2m} D\log D}{n}
      + 2\EE \left| \frac{1}{n}\sum_{i=1}^n Z_i \right|_2^p. \label{eq873}
\end{align}
Thus, we need to upper-bound
\[
\EE\norm{\frac{1}{n}\sum_{i=1}^n Z_i}_2^{p}.
\]
First, for $p=2$ we have
\[
\EE \left| \frac{1}{n}\sum_{i=1}^n Z_i \right|^2 = \frac{D-1}{n}\leq \frac{D}{n}.
\]
For general $p$, consider the random variable
\[
\tilde{Q} = \left| \frac{1}{n}\sum_{i=1}^n Z_i \right|_2^2.
\]
It is a polynomial of degree $2m$ in the log-concave variables $X_i$. Furthermore, from the case $p=2$, we know that
\[
 \EE |\tilde{Q}| = \EE \tilde{Q} \leq \frac{D}{n}.
\]
Using again the moment inequality for polynomials (Proposition \ref{prop:poly_concentration_logconcave}), we find that
\[
\left(\EE \left| \tfrac{1}{n}\sum_{i=1}^n Z_i \right|_2^p\right)^{2/p}
= \left(\EE \tilde{Q}^{p/2}\right)^{2/p}
\leq (Cp)^{2m} \EE |\tilde{Q}|
\leq \frac{(C\log n)^{2m}D}{n},
\]
where we used that $D\leq n$.
Plugging this inequality into \eqref{eq873} finally proves Lemma \ref{lem724}:
\[
 \left(\EE\norm{C_n-I_d}_{op}^p\right)^{2/p} \leq \frac{(C\log n)^{2m} D\log D}{n}.
\]
\qed

It remains to pass from an inequality on $C_n$ to a corresponding inequality for its inverse. We shall thus need an integrable bound on the probability that the smallest eigenvalue of $C_n$ is small.

\medskip
Recall that the covariance matrix is given by
\[
C_n = \frac{1}{n}\sum_{i=1}^n \tilde{Z}_i\otimes \tilde{Z}_i,
\]
where $\tilde{Z}_i=(p_k(X_i))_{0\leq k \leq D-1}$. For any $\theta\in\mathbb{S}^{D-1}$, $\tilde{Z}_i\cdot\theta$ is a polynomial of degree at most $m$ with
\[\EE \abs{\tilde{Z}_i\cdot\theta}^2=1. \]
The Carbery-Wright Theorem (\ref{thm:carbery}) implies the following small-ball property.
\begin{lemma}\label{lem:small_ball}
    For any $\theta\in \mathbb{S}^{D-1}$ and any $t\geq 0$,
    \[
    \PP\bigl(|\tilde{Z}\cdot\theta| \leq t\bigr) \leq C m\,t^{1/m}.
    \]
\end{lemma}
In the sequel, we work in the setting of Theorem \ref{thm:avg_ls}. In particular, we may assume that $n\geq C_0^m D$ for some sufficiently large constant $C_0$ and $m\geq 1$.
We control the tails of $\lambda_{\min}(C_n)$ in two regimes:
\begin{lemma}\label{lem714}
    Assume as we may that $n\geq C_0^m D$ for some sufficiently large $C_0$. Then there exist universal constants $c_0,c_1,c_2$ such that
    \[
    \PP(\lambda_{\min}(C_n) \leq e^{-c_0m}) \leq \exp\!\Bigl(-\tfrac{n}{e^{c_1m}}\Bigr).
    \]
    Furthermore, for $t\leq 1/n^2$,
    \[
    \PP(\lambda_{\min}(C_n) \leq t) \leq t^{n/16m}.
    \]
\end{lemma}
\begin{proof}
    We start by proving the first statement. Notice that
    \[
    \lambda_{\min}(C_n) = \inf_{\theta \in S^{D-2}}\frac{1}{n}\sum_{i=1}^n (\tilde{Z}_i\cdot \theta)^2.
    \]
    Now fix some unit vector $\theta$, and write $V=(\tilde{Z}\cdot\theta)^2$. Then $V$ is a non-negative random variable with $\EE V=1$ and $\EE V^2 \leq e^{cm}$ for some constant $c$, by Proposition \ref{prop:poly_concentration_logconcave}. The Paley--Zygmund inequality implies that
    \[
    \PP(V<1/2) \leq 1 - e^{-\Tilde{c}m}
    \]
    for some constant $\Tilde{c}>0$. This in turn implies that
    \begin{equation}\label{eq:negative_laplace}
        \EE e^{-mV} \leq 1-e^{-cm}
    \end{equation}
    for some constant $c>0$. We now make use of the Laplace method. Fix $\theta\in \mathbb{S}^{D-1}$ and write
    \[
    S_n = \sum_{i=1}^n (\tilde Z_i\cdot\theta)^2=\sum_{i=1}^n V_i.
    \]
    Let $c_0 = c+\log 2$, and let $t_1 = e^{-c_0m}$. By Markov's inequality,
    \begin{align*}
        \PP\Bigl(\tfrac{1}{n}S_n<t_1\Bigr)
        &= \PP\bigl(e^{-mS_n}>e^{-nmt_1}\bigr) \\
        &\leq \Bigl(\EE\bigl(e^{-mV}\bigr)e^{mt_1}\Bigr)^n\\
        &\leq \Bigl((1-e^{-cm})(1+2e^{-c_0m})\Bigr)^n\\
        &\leq (1-e^{-2cm})^n\\
        &\leq \exp\!\left(-\frac{n}{e^{2cm}}\right),
    \end{align*}
    where we used that $mt_1 \leq me^{-m\log 2}\leq 1$ and that $e^{x}\leq 1+2x$ for $x\leq 1$ and that $2e^{-c_0m}\leq e^{-cm}$.

    Now taking a union bound over a $t_1/2$-net $\cN$ of the sphere $\mathbb{S}^{D-2}$ of cardinality
    \[
    |\cN| \leq \Bigl(1+\frac{4}{t_1}\Bigr)^D
    \]
    concludes the proof of the first statement, since $n\geq C_0^mD$ for a sufficiently large chosen $C_0$.

We move to the second statement. Again, we fix some vector on the sphere and work with the same notations as before. For any $t\leq \frac{1}{n^2}$,
\begin{align*}
    \PP\Bigl(\tfrac{1}{n}S_n \leq 2t\Bigr)
    &\leq \PP(S_n\leq 2\sqrt{t})^n\\
    &\leq \PP(V\leq 2\sqrt{t})^n\\
    &= \PP(|Z\cdot\theta|\leq t^{1/4})^n\\
    &\leq C m\,t^{n/4m}\\
    &\leq t^{n/8m},
\end{align*}
where we used Carbery--Wright (Theorem \ref{thm:carbery}) on line 4 and assumed a large enough choice of $C_0$. Taking a union bound over a $t$-net of the sphere, of cardinality less than $(1+2/t)^D$, concludes the proof, again for $C_0$ large enough.

We are now in a position to prove Lemma \ref{lem688}. We use the simple observations that, for any positive matrix $M$,
\[
\norm{M^{-1}-I_d}_{op} \leq \frac{1}{\lambda_{\min}(M)}\norm{M-I_d}_{op},
\]
and that
\[
\norm{M^{-1}-I_d}_{op}\leq \max\!\left(1,\frac{1}{\lambda_{\min}(M)}\right).
\]
We abbreviate $\lambda_{\min} = \lambda_{\min}(C_n)$. Recall that
\[
m \leq \frac{\log n}{\log d}.
\]
Thus, given $c_0$ and $c_1$ the constants from Lemma \ref{lem714}, if $d$ is large enough we have
\[
e^{-c_0m}\geq \frac{1}{n^2}, \qquad e^{c_1m}\leq \sqrt{n}.
\]
We partition the probability space into three events:
\[
\cA = \{\lambda_{\min} \geq e^{-c_0m}\},
\]
\[
\cB = \Bigl\{\tfrac{1}{n^2}\leq \lambda_{\min} \leq e^{-c_0m}\Bigr\} ,
\]
\[
\cC = \{\lambda_{\min} \leq \tfrac{1}{n^2}\}.
\]
Using the previous observations and Lemma \ref{lem714},
\begin{align*}
    \EE\norm{C_n^{-1}-I_d}_{op}^p
    &= \EE\!\left[\norm{C_n^{-1}-I_d}_{op}^p\one_{\cA}\right]
     + \EE\!\left[\norm{C_n^{-1}-I_d}_{op}^p\one_{\cB}\right]
     + \EE\!\left[\norm{C_n^{-1}-I_d}_{op}^p\one_{\cC}\right]\\
    & \leq e^{pc_0m}\EE\!\left[\norm{C_n-I_d}_{op}^p\one_{\cA}\right]
        + \EE\!\left[\lambda_{\min}^{-p}\one_{\cB}\right]
        + \EE\!\left[\lambda_{\min}^{-p}\one_{\cC}\right] \\
    &\leq  e^{pc_0m}\EE\norm{C_n-I_d}_{op}^p
        + n^{2p}\exp\!\Bigl(-\tfrac{n}{e^{c_1m}}\Bigr)
        + \int_{n^{2p}}^{+\infty}\PP\!\left(\tfrac{1}{\lambda_{\min}^p} \geq u\right)du\\
    &\leq  e^{pc_0m}\EE\norm{C_n-I_d}_{op}^p
        + n^{2p}e^{-\sqrt{n}}
        + \int_{0}^{1/n^2}\PP\!\left( \lambda_{\min}\leq t\right)\frac{p}{t^{p+1}}dt \\
    &\leq  e^{pc_0m}\EE\norm{C_n-I_d}_{op}^p
        + O(e^{-n^{1/4}})
        + p\int_{0}^{1/n^2}t^{n/16m-p-1}\,dt \\
    &\leq  e^{pc_0m}\EE\norm{C_n-I_d}_{op}^p + O(e^{-n^{1/4}}).
\end{align*}
We conclude that
\[
    \left(\EE\norm{C_n^{-1}-I_d}_{op}^p\right)^{2/p} \leq \frac{(C\log n)^{2m} D\log D}{n}
\]
for some constant $C>0$, which is what we wanted.
\end{proof}

\subsubsection{Proof of Theorem \ref{thm_ls}}
We explain how to deduce Theorem \ref{thm_ls} from Theorem \ref{thm:avg_ls}. As in \ref{subsub_proj}, we set
\[
m_0 = \left\lfloor \frac{\log n}{\log d}\right\rfloor,
\]
and for any integer $1\leq p\leq m_0$,
\[
m_p = m_0-p.
\]
We again have
\[
    \frac{\binom{d+m_p}{m_p}}{n} \leq \left(\frac{m_0}{d}\right)^p.
\]
Plugging this into Theorem \ref{thm:avg_ls}, we obtain, for the choice $m=m_p$,
\begin{align*}
        \EE\normmu{f-f_{LS}}^2
        &\leq \Psi^2_\mu(m_p)  + \frac{(C\log n)^{2m_p} D\log D}{n} + \frac{8\sigma^2D}{n} \\
        &\leq \Psi^2_\mu(m_p) + (C\log n)^{2m_0+1-p}\left(\frac{m_0}{d}\right)^p + 8d\left(\frac{m_0}{d}\right)^{p}\\
        &\leq \Psi^2_\mu(m_p) +(C\log n)^{2m_0+1}\left(\frac{1}{d\log d}\right)^p + 8d\left(\frac{m_0}{d}\right)^{p}.
\end{align*}

\paragraph{First regime.}
We set $p=4$ and assume that
\[
n\leq \exp\!\left(\frac{c\log^2 d}{\log\log d}\right)
\]
for some constant $c<1$ to be determined. As a consequence, we have
\[
    \log n \cdot \log\log n \leq 2c\log^2 d,
\]
where we assume $d\geq 16$. We upper bound
\begin{align*}
(C \log n)^{2m+1}
&\leq (C \log n)^{\tfrac{4c \log d}{\log \log n} + 1} \\[6pt]
&\leq \exp\!\left( \left( \frac{4c \log d}{\log \log n} + 1 \right) \cdot \log(C \log n) \right) \\[6pt]
&\leq \exp\!\left( \left( \frac{4c \log d}{\log \log n} + 1 \right)(\log \log n + C') \right) \\[6pt]
&\leq \exp\!\left( 4c \log d + \log \log n + \frac{4c \log d}{\log \log n}\, C' + C' \right) \\[6pt]
&\leq \exp\!\left( 4c(1+C')\log d +2\log\log d +C'\right)  \\[6pt]
&\leq \exp\!\left( C' + 2\log d \right) \\[6pt]
&\lesssim d^2,
\end{align*}
where we chose
\[
c = \frac{1}{4(1+C')}.
\]
On the other hand, clearly
\[
8d\left(\frac{m_0}{d}\right)^{p} \leq 8d\left(\frac{\log d}{d}\right)^4 = o(1/d).
\]

\paragraph{Second regime.}

We want to ensure that
\[
    (C\log n)^{2m_0+1} \leq (d\log d)^{p-1}.
\]
We assume that
\[
\alpha := \frac{\log(C\log n)}{\log d} < 1/2
\]
It is enough that
\[
p\geq 1 + \frac{2m_0\log(C\log n)}{\log d} + \frac{\log(C\log n)}{\log d}.
\]
Thus, using that $\alpha<1$, it is enough that
\[
p\geq 2 + \frac{2m_0\log(C\log n)}{\log d}.
\]
As announced, we choose
\begin{align*}
    p &= 4 + \left\lfloor\frac{2m_0\log(C\log n)}{\log d}\right\rfloor \\
      &= 4 + \lfloor 2\alpha m_0 \rfloor.
\end{align*}
For that choice of $p$, since
\[
m_0=\left\lfloor\frac{\log n}{\log d}\right\rfloor \lesssim d^{1/2},
\]
we again have
\[
\sigma^2\frac{D}{n} \leq d\left(\frac{m_0}{d}\right)^4 = O(1/d).
\]

\section{The metric entropy of Lipschitz functions}\label{sec_entropy}
In the previous sections, we used low-degree multivariate polynomials to approximate and learn Lipschitz functions in high dimensions. In the Gaussian setting, when $\mu=\gamma$, for a given $\varepsilon>0$ we use that any $1$-Lipschitz function $f$ can be approximated with error at most $\varepsilon$ by a polynomial of degree at most $m$, where $m\simeq \frac{1}{\varepsilon}$. Heuristically, this approach makes sense if, at scale $\varepsilon$, the “size” of the space of polynomials of degree at most $m$ is not much larger than that of the space of Lipschitz functions. One standard way to measure size is through metric entropy. For a metric space $\left(\cX,d\right)$ we define its metric entropy as
\begin{equation}
    H_{(\cX,d)}(\varepsilon) = \log \cN_{(\cX,d)}(\varepsilon),
\end{equation}
for all $\varepsilon>0$, where $\cN_{(\cX,d)}(\varepsilon)$ is the largest cardinality of an $\varepsilon$-separated set in $(\cX,d)$. We adopt the (slightly unusual) convention of using packing numbers instead of covering numbers for our definition of metric entropy, as it will be more convenient for us.

\medskip
We provide estimates for the metric entropy of Lipschitz functions equipped with the distance
\[
d(f,g) = \norm{f-g}_{L^2(\mu)},
\]
where $\mu$ is an isotropic product log-concave measure on $\RR^d$. We denote by
\[
H_L^{\mu}(\varepsilon) = H_{(B^\mu_{Lip}, d)}(\varepsilon),
\]
where $B_{Lip}^\mu$ is the unit ball of $1$-Lipschitz functions, i.e., those $f$ such that
\[
\int f^2 \,d\mu \leq 1.
\]

\begin{theorem}\label{thm_644}
    Let $\mu$ be an isotropic product log-concave measure on $\RR^d$.
    Let $\varepsilon > 0$ satisfy
    \[
        d^{-1/4} < \varepsilon < 1.
    \]
    Then
    \[
    d^{ c/\varepsilon^2} \lesssim H_L^{\mu}(\varepsilon),
    \]
    where $c > 0$ is a universal constant.
\end{theorem}

In the case where $\mu = \gamma$, the standard Gaussian measure on $\RR^d$, we get a two-sided estimate:
\begin{corollary}\label{cor_1535}
    There exists a constant $c>0$ such that, for any $\varepsilon$ with
    \[
        d^{-1/4} < \varepsilon < 1,
    \]
    we have
    \begin{equation}\label{eq:639}
        \binom{d}{\lfloor c/\varepsilon \rfloor^2} \ \lesssim \ H_L^\gamma(\varepsilon)\ \lesssim \ \binom{d}{\lceil 4/\varepsilon \rceil^2},
    \end{equation}
    where $c > 0$ is a universal constant.
\end{corollary}

\begin{remark}
    Corollary \ref{cor_1535} extends immediately to products of isotropic log-concave measures that are Lipschitz images of the Gaussian. Indeed, if $\mu = T\#\gamma$ for some $K$-Lipschitz map $T$, then
    \[
    d_\mu(f,g) = d_\gamma(f\circ T, g\circ T)
    \]
    and thus
    \[
    H_L^\mu(\varepsilon) \leq H_L^\gamma(\varepsilon/K).
    \]
    This includes, for example, the uniform measure on the hypercube, or products of strongly log-concave densities.
\end{remark}
From Theorem \ref{thm_644}, we will deduce a slightly weaker lower bound for the general case.
\begin{theorem}\label{thm1863}
    Let $\mu$ be an isotropic log-concave probability measure on $\RR^d$. Let $\varepsilon > 0$ satisfy
    \[
        d^{-\eta} < \varepsilon < 1.
    \]
    Then
    \[
    d^{ c/\varepsilon^2} \lesssim H_L^{\mu}(\varepsilon),
    \]
    where $\eta<1/4$ and $c > 0$ are universal constants.
\end{theorem}
\medskip
We begin by proving the upper bound in Corollary \ref{cor_1535}, which essentially follows from polynomial approximation. Without loss of generality, we may assume that $d$ is large enough.
Let $\varepsilon \in (0,1)$ and let $(f_1,\dots,f_{N})$ be an $\varepsilon$-separated subset of $B_{Lip}^\gamma$. Since $\Psi_{\gamma}(m) \leq 1/(m+1)$, there exist polynomials $P_1,\dots,P_{N}$ such that
\[
\normgamma{f_i-P_i} \leq \tfrac{\varepsilon}{3},
\qquad
\deg(P_i) \leq m := \Bigl\lceil \tfrac{3}{\varepsilon} \Bigr\rceil^2.
\]
Thus, by the triangle inequality, the set $(P_1,\dots,P_{N})$ is $\varepsilon/3$-separated; indeed, for $i\neq j$,
\[
\normgamma{P_i-P_j}\geq \tfrac{\varepsilon}{3}.
\]
In fact, for any $i$, the polynomial $P_i$ is the truncated Hermite expansion of the $1$-Lipschitz function $f_i$:
\[
P_i = \sum_{\abs{\alpha}\leq m} \langle f_i, H_\alpha \rangle \, H_\alpha.
\]
In particular,
\[
\normgamma{P_i} \leq \normgamma{f_i} \leq 1.
\]
Hence $P_i$ lies in the unit ball of $\cP_{d,m}$, equipped with
the norm $\normgamma{\cdot}$. As before, we set
\[
D = \binom{d+m}{m}
\]
for the dimension of that space. We thus have the standard packing bound
\[
N \leq \left(1+ \tfrac{6}{\varepsilon}\right)^D \leq \left(\tfrac{7}{\varepsilon}\right)^D.
\]
Let $m_2 := \bigl\lceil \tfrac{4}{\varepsilon} \bigr\rceil^2$ and $D_2 := \binom{d+m_2}{m_2}$. For $d$ large enough,
\[
D\log\!\Bigl(\tfrac{7}{\varepsilon}\Bigr) \leq D_2,
\]
so that
\[
N \leq e^{D_2}.
\]
Finally,
\begin{align*}
    \frac{\binom{d+m_2}{m_2}}{\binom{d}{m_2}}
    &= \frac{(d+m_2)!(d-m_2)!}{d!^2}\\
    &= \frac{(d+m_2)(d+m_2-1)\cdots(d+1)}{d(d-1)\cdots(d-m_2+1)} \\
    &\leq \left(\frac{d+m_2}{d-m_2}\right)^{m_2}
     = \left(1+\frac{2m_2}{d-m_2}\right)^{m_2}
     \leq \left(1+\frac{4}{m_2}\right)^{m_2}
     \leq e^4.
\end{align*}
This concludes the proof of the upper bound:
\[
\log N \;\leq\; D_2 \;\lesssim\; \binom{d}{\lceil 4/\varepsilon \rceil^2}.
\]
The constant $4$ is not optimal and can in fact be reduced essentially to $2$.
\subsection{Lower bound}
For the lower bound, given $\varepsilon>\tfrac{1}{d^{1/4}}$,
our strategy is to begin with a $\tfrac{1}{2}$-separated set of polynomials of degree at most $m$, with $m\simeq \tfrac{1}{\varepsilon^2}$, and from it construct an $\varepsilon$-separated set of Lipschitz functions. By convolving $\mu$ with a tiny Gaussian of variance tending to zero, it is not difficult to show that we may assume that $\mu$ has and positive density on the whole $\RR^n$. From now on we fix such an isotropic product log-concave measure $\mu$ and denote by
\[
(T_t)_{t\geq0}
\]
the associated Langevin semigroup. One possible way of transforming a polynomial $P$ into a Lipschitz function $f_P$ is to set
\begin{equation}\label{eq732}
    f_P = T_t(P|_\lambda),
\end{equation}
for some $\lambda,t>0$, where $P|_\lambda$ denotes the truncation
\[
P|_\lambda(x) = P(x)\,1_{\{|P(x)|\leq \lambda\}}.
\]
By construction $P|_\lambda$ is bounded by $\lambda$, thus by Lemma \ref{lem1162},
\[
f_P \text{ is } \tfrac{\lambda}{\sqrt{t}}\text{-Lipschitz}.
\]
We shall choose $t$ and $\lambda$ so that the $L^2$ norm of $f_P$ is not too different from that of $P$. More precisely, we would like to ensure that if $P$ and $Q$ are two polynomials of degree at most $m$ such that
\[
\normmu{P-Q}\geq \tfrac{1}{2},
\]
then
\begin{equation}\label{eq740}
    \normmu{f_P-f_Q}\geq c>0.
\end{equation}
If we can ensure \eqref{eq740} for any pair of polynomials $P,Q$ in a $\tfrac{1}{2}$-separated set of $\cP_{d,m}$, then we will have constructed a $c$-separated set of Lipschitz functions with Lipschitz constant $\tfrac{\lambda}{\sqrt{t}}$. Equivalently, a $\tfrac{c\sqrt{t}}{\lambda}$-separated set of $1$-Lipschitz functions.

Let us discuss what values of $t$ and $\lambda$ might ensure \eqref{eq740}. At this heuristic level, it is helpful to consider the case $\mu=\gamma$. In this case, the Langevin semigroup is the Ornstein–Uhlenbeck semigroup, which acts diagonally on Hermite polynomials:
\[
T_tH_\alpha = e^{-t|\alpha|}H_\alpha.
\]
Thus, by decomposing a polynomial $P$ of degree at most $m$ into the orthonormal Hermite basis,
\[
P = \sum_{|\alpha|\leq m} P_\alpha H_\alpha,
\]
we obtain
\begin{equation}\label{eq748}
\normgamma{T_tP}^2 = \sum_{|\alpha|\leq m} e^{-2t|\alpha|}P_\alpha^2 \;\geq\; e^{-2tm}\sum_{|\alpha|\leq m} P_\alpha^2 = e^{-2tm}\normgamma{P}^2.
\end{equation}
Although we will apply $T_t$ to the truncated polynomial $P|_\lambda$ rather than to $P$ itself, using the fact that $T_t$ is a contraction in $L^2(\gamma)$ we may write
\begin{equation}
    \normgamma{T_t(P|_\lambda)} \;\geq\; \normgamma{T_tP} - \normgamma{T_t(P-P|_\lambda)}
    \;\geq\; e^{-tm}\normgamma{P} - \normgamma{P-P|_\lambda}.
\end{equation}
Thus, if we choose $t$ of order $1/m$, we must choose $\lambda$ large enough so that
\[
\normgamma{P-P|_\lambda}
\]
is sufficiently small. The issue is that for an arbitrary degree-$m$ polynomial $P$ with
\[
\normgamma{P}=1,
\]
if one wants to truncate at some level $\lambda>0$ so that
\[
\normgamma{P-P|_\lambda}\leq \tfrac{1}{10},
\]
one may have to take
\[
\lambda \geq e^{cm},
\]
which is too large for our purposes. Indeed, the fourth moment of $P$ may be as large as
\[
\EE P^4(G) \geq e^{cm}
\]
for some constant $c>0$. This can already be seen in dimension one by considering the degree-$m$ monomial.

\medskip
We resolve this issue by considering random degree-$m$ polynomials, for which we show that, with positive probability, it suffices to take
\[
\lambda = \lambda_0>0,
\]
a constant independent of $m$. Moving away from the Gaussian setting, we also show that for such random polynomials the Langevin semigroup does not “kill’’ the $L^2$ norm too quickly.

\subsubsection{Random multilinear polynomials}
We restrict our attention to polynomials which are linear combinations of degree-$m$ multilinear monomials. Let
\[
D_0 = \binom{d}{m}.
\]
Write $\cS_{d,m}$ for the collection of all subsets of $\{1,\ldots,d\}$
of cardinality $m$.
We define a polynomial
\begin{equation}\label{eq775}
    P_\theta  = \sum_{\alpha \in \cS_{d,m} } \theta_\alpha \prod_{i\in \alpha} X_i
    = \sum_{\alpha \in \cS_{d,m}} \theta_\alpha X_{\alpha},
\end{equation}
for a given vector $\theta \in \RR^{D_0} \cong \RR^{\cS_{d,m}}$, where we write
\[
X_{\alpha} = \prod_{i\in \alpha}X_i.
\]
Our intuition is that for a random $\theta$, the value distribution of $P_{\theta}(X)$ should be roughly Gaussian.

\begin{lemma}\label{lem_780}
    Assume that $m^2\leq d$ and let $P_\theta$ be defined by \eqref{eq775}. Let $\theta \in \RR^{D_0}$ be a Gaussian random vector of mean zero and covariance $(1/D_0) \cdot I_{D_0}$.
    Then the expected fourth moment of $P_\theta$ is bounded by
    \begin{equation}
        \EE_\theta \left[ \EE P_\theta^4(X)  \right] \leq 3 + \frac{Cm^2}{d} \leq C_0,
    \end{equation}
    where $X\sim\mu$, $C$ is a universal constant, and $C_0 = 3+C$.
\end{lemma}

\begin{proof}
We expand
\begin{align*}
\EE_\theta \left[ \EE\, P_\theta^4(X) \right]
&= \EE \sum_{\alpha_1,\alpha_2,\alpha_3,\alpha_4}
   \theta_{\alpha_1}\theta_{\alpha_2}\theta_{\alpha_3}\theta_{\alpha_4}
   X_{\alpha_1}X_{\alpha_2}X_{\alpha_3}X_{\alpha_4} \\
&= \sum_{\alpha} \EE[\theta_\alpha^4]\,\EE[X_{\alpha}^4]
   + 3\sum_{\alpha\neq\beta} \EE[\theta_\alpha^2]\,\EE[\theta_\beta^2]\,
     \EE[X_{\alpha}^2X_{\beta}^2] \\
&= \frac{3}{D_0^2}\sum_\alpha \EE[X_{\alpha}^4]
   + \frac{3}{D_0^2}\sum_{\alpha\neq\beta} \EE[X_{\alpha}^2X_{\beta}^2]\\
&= \frac{3}{D_0^2}\sum_{\alpha,\beta} \EE[X_{\alpha}^2X_{\beta}^2].
\end{align*}
Here we used that
\[
\EE[\theta_{\alpha_1}\theta_{\alpha_2}\theta_{\alpha_3}\theta_{\alpha_4}] \neq 0
\]
if and only if all four indices are equal (giving the first term), or if they form two distinct pairs (three such pairings, giving the second term).
Now, for $\alpha,\beta \in \cS_{d,m}$,
\[
\EE\bigl[X_{\alpha}^2X_{\beta}^2\bigr]
= \EE\!\left(\prod_{i\in \alpha \cap \beta}X_i^4
              \prod_{i\in \alpha \cup \beta \setminus(\alpha \cap \beta)}X_i^2\right)
\leq 9^{|\alpha\cap\beta|},
\]
using that for any centered log-concave random variable $X$,
\[
\EE X^4 \leq 9\,\EE X^2,
\]
see e.g. \cite[Theorem 1.4]{yam}.
Thus
\begin{align*}
    \frac{3}{D_0^2}\sum_{\alpha,\beta} \EE[X_{\alpha}^2X_{\beta}^2]
    &\leq \frac{3}{D_0^2}\sum_{\alpha,\beta} 9^{|\alpha\cap\beta|} \\
    &= 3\,\EE 9^{|\alpha\cap\beta|} \\
    &= 3\,\EE 9^{|\alpha\cap \{1,\ldots,m\}|},
\end{align*}
where in the last line we denote by $\alpha$ and $\beta$ two independent uniform random subsets of $\{1,\ldots,d\}$ of size $m$, and used invariance under any bijection. The random variable $|\alpha\cap \{1,\ldots,m\}|$ follows a hypergeometric distribution:
\[
|\alpha\cap\{1,\ldots,m\}| \sim \mathrm{Hypergeometric}(d,m,m).
\]
It is well known that $\mathrm{Hypergeometric}(N,K,n)$ is stochastically dominated by $\mathrm{Binomial}(n,K/N)$ (and Hoeffding \cite{hoeffding1963probability} even proved that the same domination also holds in the convex order). Thus, for any increasing or convex $f$,
\[
\EE f\bigl(|\alpha\cap[1,m]|\bigr) \leq \EE f(Z),
\]
where $Z\sim \mathrm{Binomial}(m,p=m/d)$. In particular,
\begin{align*}
    \EE 9^{|\alpha\cap[1,m]|}
    &\leq \EE 9^Z
    = (1+8p)^m \\
    &= \left(1+\frac{8m}{d}\right)^m\\
    &\leq \exp(8m^2/d)
    \;\leq\; 1 + C_0 m^2/d,
\end{align*}
where one may take $C_0 = e^{8}-1$. This concludes the proof of Lemma \ref{lem_780}.
\end{proof}
We have established that, on average, the $4$-th moment of the random multilinear polynomials is bounded. We now need an argument to show that their $L^2$ norm does not decay too quickly along the Langevin semigroup. We will use Lemma~\ref{lem309} from Section~\ref{sec_logconcave}, which states that for any
$f\in L^2(\mu)$ with square integrable gradient,
\begin{equation}\label{eq2070}
    \norm{T_t^\mu f}_{L^2(\mu)}^2 \;\geq\; \norm{f}_{L^2(\mu)}^2 - 2t\int \norm{\nabla f}^2\,d\mu.
\end{equation}
As before we denote the multilinear polynomials by
\[
X_\alpha = \prod_{i\in\alpha}X_i.
\]
Clearly, for any $1\leq i \leq d$
\[
\partial_i X_\alpha =
\begin{cases}
0, & \text{if } i \notin \alpha, \\
X_{\alpha \setminus \{i\}}, & \text{otherwise.}
\end{cases}
\]
In particular, since $\mu$ is a product measure, for a fixed $i$, the family $(\partial_i X_\alpha)$ is orthonormal in $L^2(\mu)$. For any $\theta\in\RR^{D_0}$,
\begin{align*}
    \int \norm{\nabla P_\theta}^2 d\mu
    &= \int \sum_{i=1}^d (\partial_{i}P_\theta)^2 d\mu\\
    &= \sum_{i=1}^d \sum_{\substack{\alpha\subset\{1,\ldots,d\}\\ \abs{\alpha}=m}} \int \theta_\alpha^2(\partial_i X_\alpha)^2 d\mu \\
    &= \sum_{\substack{\alpha\subset\{1,\ldots,d\}\\ \abs{\alpha}=m}} \sum_{i=1}^d \theta_\alpha^2\one_{i\in \alpha} \\
    &= m\norm{\theta}_2^2 \numberthis\label{eq1806}
\end{align*}

We are now in a position to prove the lower bound of Theorem \ref{thm_644}. Let $N$ be an integer to be chosen later, and let $\theta_1,\dots,\theta_N$ be i.i.d random vectors with distribution
\[
\theta_i \sim \cN\!\left(0,\tfrac{1}{D_0}I_{D_0}\right).
\]
Let $P_{\theta_1},\dots,P_{\theta_N}$ be the corresponding polynomials defined by \eqref{eq775}. For any $1\leq i,j \leq N$ we have
\[
\normmu{P_i} = \norm{\theta_i}_2
\quad\text{and}\quad
\normmu{P_i-P_j} = \norm{\theta_i-\theta_j}_2.
\]
Furthermore, for any $i\neq j$, the random vector $\theta_i -\theta_j$ is again Gaussian with covariance $\tfrac{2}{{D_0}}I_{D_0}$. By Gaussian concentration for Lipschitz functions and a union bound, we obtain
\begin{align*}
    \PP\!\left(\exists\, 1\leq i\neq j \leq N : \ \norm{\theta_i-\theta_j}_2 \leq 1\right)
    &\leq N^2\, \PP\!\left(\norm{\tfrac{\sqrt{2}}{\sqrt{{D_0}}}G}_2 \leq 1\right) \\
    &\leq N^2\,\PP\!\left(\norm{G} \leq \EE\norm{G} - \tfrac{\sqrt{{D_0}}}{4}\right) \\
    &\leq N^2 e^{-{D_0}/32},
\end{align*}
where $G\sim\cN(0,I_{D_0})$, and where we used that
\[
\sqrt{{D_0}}-1 \;\leq\; \EE\norm{G}_2.
\]
We also have the tail bound
\[
\PP\!\left(\exists\, 1\leq i\leq N : \ \norm{\theta_i}^2_2 \geq 2\right) \;\leq\; N e^{-{D_0}/4}.
\]
We choose $N = e^{{D_0}/128}$, and define the events
\[
\cA = \{ (P_{\theta_1},\dots,P_{\theta_N}) \ \text{is a $1$-separated set in } L^2(\mu) \},
\]
and
\[
\cB = \{ \norm{\theta_i}^2_2\leq 2 \qquad \forall 1\leq i\leq N \}.
\]
From the two previous inequalities we deduce that
\[
\PP(\cA\cap \cB) \geq 1 - 2e^{-{D_0}/64} \;\geq\; \tfrac{3}{5},
\]
say. On the other hand, by Lemma \ref{lem_780} and Markov's inequality, we have that
\[
p= \PP\bigl(\EE_X P_\theta^4(X) \leq 2C_0\bigr) \geq 1/2,
\]
where $C_0$ is the constant from Lemma \ref{lem_780}. Thus, roughly half of the $(P_{\theta_i})_{1\leq i \leq N}$ will enjoy a nice bound on their fourth moment. That is, define
\[
N_1 = \#\{1\leq i \leq N : \ \EE_X P_{\theta_i}^4(X) \leq 2C_0\} \sim \mathrm{Binomial}(N,p).
\]
The median of a Binomial with parameters $(N,p)$ is greater than $\lfloor Np \rfloor$. Thus, with probability $1/2$, we have
\[
N_1 \geq \lfloor pN\rfloor \geq N/3.
\]
The event
\[
\cD = \cA\cap\cB\cap\{N_1 \geq N/3\}
\]
has positive probability, greater than $0.1$. For such a realization we find polynomials
\[
(P_i)_{1\leq i \leq N_1}
\]
that form a $1$-separated set of $B_{L^2(\mu)}(0,\sqrt 2)$ of cardinality
\[
N_1\geq N/3 = e^{{D_0}/128}/3 \geq e^{{D_0}/256},
\]
and satisfy
\begin{equation}\label{eq1862}
    \EE_X P_i^4(X) \leq 2C_0 \quad \forall 1\leq i \leq N_1,
\end{equation}
\begin{equation}\label{eq1865}
    \normmu{P_i}^2 = \norm{\theta_i}_2^2 \leq 2,
\end{equation}
\begin{equation}\label{eq1869}
    \normmu{\nabla P_i}^2 = m\norm{\theta_i}_2^2 \leq 2m.
\end{equation}
As described above, we set
\begin{equation}\label{eq854}
    f_i = T_{t}(P_i|_\lambda)
\end{equation}
with
\[
t = \tfrac{1}{32m},
\]
and $\lambda$ to be chosen later. Then $f_i$ is Lipschitz with constant
\begin{equation}\label{eq1879}
    \frac{\lambda}{\sqrt{t}}= 4\sqrt{2}\lambda \sqrt{m}.
\end{equation}
Secondly, $T_t$ is a contraction in $L^2(\mu)$, so
\[
\normmu{f_i} \leq \normmu{P_i|_\lambda} \leq \normmu{P_i}\leq 2.
\]

Let us verify that $(f_i)_{1\leq i\leq N_1}$ is separated. Let $i\neq j$, using the triangle inequality, \eqref{eq2070}, \eqref{eq1806} and \eqref{eq1869}, we get
\begin{align*}
    \normmu{f_i-f_j} &= \normmu{T_t(P_i|_\lambda - P_j|_\lambda)}\\
    & \geq \normmu{T_t(P_i-P_j)} - \normmu{T_t(P_i|_\lambda - P_j|_\lambda - P_i + P_j)} \\
    &\geq \left(\normmu{P_i-P_j}^2 - 2t\int \norm{\nabla\left(P_i-P_j\right)}^2d\mu \right)^{1/2} - \normmu{P_i - P_i|_\lambda} -\normmu{P_j - P_j|_\lambda} \\
    &\geq \left(1 - 4t\bigl(\normmu{\nabla P_i}^2 + \normmu{\nabla P_j}^2\bigr)  \right)^{1/2} - \normmu{P_i- P_i|_\lambda} -\normmu{P_j- P_j|_\lambda} \\
    &\geq \left(1 - 16tm \right)^{1/2} - \normmu{P_i- P_i|_\lambda} -\normmu{P_j- P_j|_\lambda}\\
    &\geq \tfrac{1}{\sqrt2} - \normmu{P_i- P_i|_\lambda} -\normmu{P_j- P_j|_\lambda}. \numberthis \label{eq1223}
\end{align*}
It remains to upper-bound
$$\normmu{P_i- P_i|_\lambda}\qquad \forall 1\leq i \leq N_1.$$
We fix some $1\leq i\leq N$, let
$$\cE = \{\abs{P_i}\geq \lambda\}.$$
Using Markov's inequality and \eqref{eq1865},
\begin{equation}\label{eq1894}
    \PP(\cE) \leq \frac{\normmu{P_i}^2}{\lambda^2} \leq \frac{2}{\lambda^2}.
\end{equation}
Using Cauchy-Schwarz and \eqref{eq1862},
\begin{align*}
    \normmu{P_i- P_i|_\lambda}&=\normmu{P_i\ind_\cE}\\
    &\leq \norm{P_i}_{L_4(\mu)}\norm{\ind_A}_{L_4(\mu)} \\
    &\leq \left(\frac{2C_0}{\lambda^2}\right)^{1/4}.
\end{align*}
We choose
$$\lambda = 16\sqrt{2}\sqrt{C_0}$$
and we find that for all $1\leq i \leq N_1$,
$$\normmu{P_i- P_i|_\lambda} \leq \frac{1}{4}.$$
Plugging this back into \eqref{eq1223}, we arrive at
\begin{equation}
    \normmu{f_i-f_j} \geq \frac{1}{\sqrt{2}}-\frac{1}{2} \geq \frac{1}{5}
\end{equation}
for all $1\leq i\neq j \leq N_1$.
Setting
$$\Tilde{f_i} = \frac{1}{4\sqrt{2}\sqrt{m}\lambda}f_i = \frac{1}{128\sqrt{C_0}\sqrt{m}}f_i$$
We have constructed a family of $1$-Lipschitz functions which is $\frac{\tilde{C}}{\sqrt{m}}$-separated and has cardinality
$$N_1 \geq e^{D_0/256}$$
where
$$D_0 = \binom{d}{m}.$$
In other words, for a given $\varepsilon>0$, setting
$$m = \lfloor \frac{\tilde{C}^2}{\varepsilon^2} \rfloor$$
we constructed an $\varepsilon$-separated set of cardinality of cardinality $N_1$ with
$$\log N_1 \gtrsim \binom{d}{m} \geq \binom{d}{\frac{c}{\varepsilon^2}}$$
for some constant $c>0$, which is what we wanted to prove.

\subsubsection{The general case}

We explain how to deduce the general isotropic case from the case of the product measure. It is well-known that lower-dimensional marginals of an isotropic log-concave probability measure are approximately Gaussian. The following precise statement was proved in \cite{EK_pointwise}:

\begin{theorem}\label{thm_pointwise} Let $\mu$ be an isotropic log-concave probability measure in $\RR^d$. Then there exists a subspace $E \subseteq \RR^d$ of dimension $k\geq d^{\eta_0}$ such that
\begin{equation}\label{eq2246}
    \abs{p_E(x)-q_{\gamma_E}(x)}\leq \frac{C}{k}q_{\gamma_E}(x) \qqquad \text{for all }\abs{x}\leq k
\end{equation}
where $C,\eta_0$ are universal constants, $p_E$ is the density of the marginal $\mu_E$ of $\mu$ on $E$ and $q_{\gamma_E}(x)
= (2 \pi)^{-k/2} e^{-|x|^2/2}$  is the density of a standard Gaussian on $E$, which we denote by $\gamma_E$.
\end{theorem}
The estimate \eqref{eq2246} implies that $p_E$ is very close to $\gamma_E$ on a ball of radius $k$, while most of the mass of $\mu_E$, or $\gamma_E$, is concentrated in a ball of radius only $\simeq \sqrt{k}$. This implies in particular that the $L^2$ norm of a Lipschitz function does not change much when swithcing from $\mu_E$ to $\gamma_E$. Indeed, let $g$ be a $1$-Lipschitz function. Then,
\begin{align*}
    \int_{E} g^2 d\mu &\geq \int_{\abs{x}\leq k} g^2 d\mu\\
    & = \int g^2 d\gamma_E -\int_{\abs{x}\leq k} g^2 (d\gamma_E-d\mu_E) - \int_{\abs{x}>k} g^2 d\gamma_E\\
    &\geq \int g^2 d\gamma_E - \frac{C}{k}\int g^2 d\gamma_E - \left(\int g^4 d\gamma_E\right)^{1/2}\PP\left(\abs{G_E}\geq k\right)^{1/2}\\
    &\geq \int g^2 d\gamma_E\left(1-\frac{C}{k}\right) -C_1e^{-k}\numberthis \label{eq2257}
\end{align*}
where $C_1$ is a universal constant. In the last line we have used concentration of the norm of a standard $k$-dimensional Gaussian:
$$\gamma_k\left(\abs{x}\geq \sqrt{k}+t\right) \leq e^{-t^2/2}.$$
and that for all $1$-Lipschitz functions $g$,
$$ \left(\int g^4d\gamma \right)^{1/2}\leq \int g^2d\gamma + \tilde{C_1}$$
for some constant $\tilde{C_1}$.

\medskip
Let $1>\varepsilon>\frac{1}{k^{1/4}}$. By Theorem \ref{thm_644}, we can find $1$-Lipschitz functions $f_1,\ldots,f_N$ such that for all $i\neq j$
$$\| f_i-f_j \|_{L^2(\gamma_E)}\geq \varepsilon$$
and
$$\log N \gtrsim k^{\frac{c}{\varepsilon^2}}\gtrsim d^{\frac{c'}{\varepsilon^2}}.$$
We now apply \eqref{eq2257} to the $1$-Lipschitz function $g=\tfrac{1}{2}(f_i-f_j)$. For large enough $d$, we get
\begin{align*}
    \| f_i-f_j \|_{L^2(\mu_E)}^2&\geq \| f_i-f_j \|_{L^2(\gamma_E)}^2\left(1-\frac{C}{k}\right) - 4C_1e^{-k}\\
    &\geq \varepsilon^2\left(1-\frac{C}{k}\right) - 4C_1e^{-k}\\
    &\geq \frac{\varepsilon^2}{4}
\end{align*}
where we assume that $d$ is large enough so that $\frac{C}{k}\leq \frac{1}{2}$ and $4C_1e^{-k}\leq\frac{1}{4\sqrt{k}}\leq \frac{\varepsilon^2}{4}$. Thus the functions $f_1,\ldots,f_N$ are $\varepsilon/2$ separated in $L^2(\mu_E)$. Equivalently, the functions $f_1\circ\Pi_E,\ldots,f_N\circ\Pi_E$ are $\varepsilon/2$ separated in $L^2(\mu)$, where $\Pi_E: \RR^d \rightarrow E$ is the orthogonal projection operator. Thus for all $\varepsilon>d^{-\eta_0/4}$,
$$H_L^\mu(\varepsilon) \geq \log N\gtrsim d^{\frac{c'}{\varepsilon^2}}.$$
This proves the general case, with $\eta = \eta_0/4.$

\subsection{A Minimax lower bound for learning Lipschitz functions}
We now go back to the learning problem
\begin{equation}
    Y_i = f(X_i) + \sigma Z_i \quad i=1,\ldots,n.
\end{equation}
and we prove the minimax lower bound announced in the Introduction, Corollary~\ref{corollary_fano}, which we restate below for the reader’s convenience.

\begin{corollary*}[Corollary \ref{corollary_fano}]
     Let $\mu$ be an isotropic log-concave measure on $\RR^d$. Assume that the noise satisfies
     \[
         n^{-\kappa} \leq \sigma^2 \leq n
     \]
     for some constant $\kappa>0$. There exists a universal constant $c>0$ such that if
     \[
         n \leq e^{\frac{cd^{2\eta}\log d}{\kappa}},
     \]
     the minimax risk is lower bounded as
     \begin{equation}\label{eq2302}
         \cR^*_{n,d}\gtrsim (1+\kappa)\frac{\log n}{\log d}.
     \end{equation}
      Moreover, if additionally $\mu$ is a product measure, then the lower bound \eqref{eq2302} holds in the range
     \[
         n \leq e^{\frac{c\sqrt{d}\log d}{\kappa}}.
     \]
\end{corollary*}

A typical way of establishing lower bounds for a learning problem is to reduce it to a multiple hypothesis testing problem and apply information-theoretic methods. This is known as Fano’s method. More precisely, we shall use the Yang–Barron version, which requires computing entropy estimates in the Kullback–Leibler divergence for the collection of random variables
\[
\mathcal{D}_n = \{((X_1,Y_1),\ldots,(X_n,Y_n)) : f\in B_2(\mathrm{Lip},0,1)\}
= \{(X,Y_f) : f\in B_2(\mathrm{Lip},0,1)\},
\]
namely, the collection of all possible random vectors that we may observe, indexed by our function class, the $1$-Lipschitz functions with bounded $L^2(\mu)$ norm. Let $P$ and $Q$ be two probability measures on $\RR^d$ such that $P$ is absolutely continuous with respect to $Q$. The Kullback–Leibler divergence from $P$ to $Q$, denoted by $\DKL(P \,\|\, Q)$, is defined as
\[
\DKL(P \,\|\, Q) := \int_{\RR^d} \log\!\left( \frac{dP}{dQ} \right) \, dP,
\]
where $\frac{dP}{dQ}$ denotes the Radon–Nikodym derivative of $P$ with respect to $Q$.
For $\varepsilon>0$, let
\[
\tilde{\cN}(\cD_n, \varepsilon, \DKL)
\]
be the minimal size of an $\varepsilon$-net of $\cD_n$ with respect to $\DKL$, and set
\[
\tilde{\mathrm{H}}(\cD_n, \varepsilon, \DKL) := \log \tilde{\cN}(\cD_n, \varepsilon, \DKL).
\]
Here the entropy is defined via covering numbers; we use a tilde to emphasize the distinction from the earlier convention adopted for $H_L$, which was based on packing numbers.

\medskip
The Yang–Barron method is summarized in the next lemma; see \cite{wainwright2019high} for background.

\begin{lemma}
    Let $\varepsilon>0$ be such that
    \begin{equation}\label{eq1251}
        \varepsilon^2 \geq \tilde{\mathrm{H}}(\cD_n, \varepsilon, \DKL),
    \end{equation}
    and $\delta>0$ be such that
    \begin{equation}\label{eq1255}
        H_L(\delta) \geq 4\varepsilon^2 + 2\log 2.
    \end{equation}
    Then, the minimax risk using $n$ samples is lower bounded as
    \begin{equation}
        \inf_{\hat{f}}\sup_{f\in B_2(\mathrm{Lip},0,1)}\EE\norm{f-\hat{f}}_{L^2(\mu)}^2 \gtrsim \delta^2.
    \end{equation}
\end{lemma}

\begin{proof}[Proof of Corollary~\ref{corollary_fano}]
    We first compute the metric entropy of $\cD_n$ equipped with the Kullback--Leibler divergence. Let $f_1$ and $f_2$ be two Lipschitz functions. For $k=1,2$, the vector $Y_{f_k}$ decomposes as
    \[
    Y_{f_k} = f_k(X) + G_k,
    \]
    where $f_k(X) = (f_k(X_1),\ldots,f_k(X_n))$ and $G_k\sim\cN(0,\sigma^2I_n)$ are independent Gaussian vectors. Conditioning on $X$, $Y_{f_1}$ and $Y_{f_2}$ are independent Gaussians with means $f_1(X)$ and $f_2(X)$, and covariance $\sigma^2I_n$. It follows that
    \begin{align*}
        \DKL\!\left((X,Y_{f_1}) \,\|\, (X,Y_{f_2})\right)
        &= \EE\!\left[\DKL\bigl(Y_{f_1}\mid X \,\|\, Y_{f_2}\mid X\bigr)\right] \\
        &= \EE\!\left(\frac{1}{2\sigma^2}\sum_{i=1}^n\bigl(f_1(X_i)-f_2(X_i)\bigr)^2\right) \\
        &= \frac{n}{2\sigma^2}\,\norm{f_1-f_2}_{L^2(\mu)}^2.
    \end{align*}
    In particular, choosing $f_1=0$, the radius of $\cD_n$ in Kullback--Leibler divergence is at most
    \[
        \frac{n}{2\sigma^2}.
    \]
    Set
    \[
        \varepsilon^2 = \frac{n}{2\sigma^2},
    \]
    which trivially ensures \eqref{eq1251}. To satisfy \eqref{eq1255}, we require
    \[
        H_L(\delta) \gtrsim \frac{n}{\sigma^2}.
    \]
    By Theorem \ref{thm_644}, provided that $\delta \ge d^{-\eta}$
    in the general case (respectively, $\delta \ge d^{-1/4}$ in the product case),  it suffices that
    \[
        d^{\,c/\delta^{2}} \ge \frac{n}{\sigma^2},
    \]
    i.e.
    \[
        \delta^{2} \lesssim \frac{\log d}{\log(n/\sigma^2)}.
    \]
    Using $\sigma^2 \ge n^{-\kappa}$, we have $\log(n/\sigma^2) \ge (1+\kappa)\log n$, so we may take
    \[
        \delta^{2} = \frac{c}{1+\kappa}\,\frac{\log d}{\log n}
    \]
    for some $c>0$. The applicability condition $\delta \ge d^{-\eta}$ (respectively $\delta \ge d^{-1/4}$) amounts to
    \[
        \delta^2 \ge d^{-2\eta}
        \quad\Longleftrightarrow\quad
        \frac{c}{1+\kappa}\,\frac{\log d}{\log n} \;\ge\; d^{-2\eta},
    \]
    (respectively $\frac{c}{1+\kappa}\,\frac{\log d}{\log n} \;\ge\; d^{-1/2}$)
    which holds whenever
        \[
        n \le \exp\!\left(\frac{c\,d^{2\eta}\,\log d}{1+\kappa}\right).
    \]
    Respectively,
    \[
        n \le \exp\!\left(\frac{c\,\sqrt{d}\,\log d}{1+\kappa}\right).
    \]
    This yields the stated bounds.
\end{proof}


\end{document}